\documentclass[11pt]{article}

\usepackage[a4paper,margin=29mm]{geometry}
\usepackage{amsmath,amssymb,amsthm,mathtools}
\usepackage{array}
\usepackage{bm}
\usepackage{enumitem}
\usepackage{cite}
\usepackage{hyperref}
\usepackage[nameinlink,capitalize]{cleveref}
\usepackage{authblk}
\usepackage{microtype}

\hypersetup{
 colorlinks=true,
 linkcolor=blue,
 citecolor=blue,
 urlcolor=blue,
 pdftitle={Pathwise stability for one-dimensional SDEs driven by Brownian motion and a symmetric stable process},
 pdfauthor={Takuya Nakagawa and Ryoichi Suzuki},
 pdfsubject={Quantitative pathwise stability for mixed Brownian-stable stochastic differential equations},
 pdfkeywords={symmetric stable process, Brownian motion, stochastic differential equation, pathwise stability, Komatsu method}
}

\newtheorem{theorem}{Theorem}[section]
\newtheorem{proposition}[theorem]{Proposition}
\newtheorem{lemma}[theorem]{Lemma}
\newtheorem{corollary}[theorem]{Corollary}
\newtheorem{remark}[theorem]{Remark}
\newtheorem{assumption}[theorem]{Assumption}
\newtheorem{definition}[theorem]{Definition}

\crefname{assumption}{Assumption}{Assumptions}
\Crefname{assumption}{Assumption}{Assumptions}
\crefname{proposition}{Proposition}{Propositions}
\Crefname{proposition}{Proposition}{Propositions}
\crefname{theorem}{Theorem}{Theorems}
\Crefname{theorem}{Theorem}{Theorems}
\crefname{lemma}{Lemma}{Lemmas}
\Crefname{lemma}{Lemma}{Lemmas}
\crefname{corollary}{Corollary}{Corollaries}
\Crefname{corollary}{Corollary}{Corollaries}

\newcommand{\R}{\mathbb{R}}
\newcommand{\E}{\mathbb{E}}
\newcommand{\Pbb}{\mathbb{P}}
\newcommand{\F}{\mathcal{F}}
\newcommand{\1}{\mathbf{1}}
\newcommand{\dd}{\,\mathrm{d}}
\newcommand{\md}{\mathrm{d}}
\newcommand{\wt}{\tilde}
\newcommand{\eps}{\varepsilon}
\newcommand{\sgn}{\mathrm{sgn}}
\newcommand{\supp}{\mathrm{supp}}

\newcommand{\RieszK}{\mathcal K_\alpha}

\title{Pathwise stability for one-dimensional SDEs driven by Brownian motion and a symmetric stable process}
\author[1]{Takuya Nakagawa\thanks{Email: \href{mailto:takuya.nakagawa73@gmail.com}{takuya.nakagawa73@gmail.com}. ORCID: \href{https://orcid.org/0000-0003-0379-5833}{0000-0003-0379-5833}.}}
\author[2]{Ryoichi Suzuki\thanks{Corresponding author. Email: \href{mailto:rsuzukimath@gmail.com}{rsuzukimath@gmail.com}. ORCID: \href{https://orcid.org/0000-0001-9979-1882}{0000-0001-9979-1882}.}}
\affil[1]{Graduate School of Science, Kyoto University, Yoshida-honmachi, Kyoto 606--8501, Japan}
\affil[2]{Department of Business Economics, School of Management, Tokyo University of Science, 1--11--2 Fujimi, Chiyoda-city, Tokyo 102--0071, Japan}
\date{}

\begin{document}
\maketitle

\begin{abstract}
We prove quantitative pathwise stability estimates for one-dimensional stochastic differential equations driven by a common Brownian motion and a common symmetric $\alpha$-stable process, where $\alpha\in(1,2)$. The comparison is made under a synchronous coupling and is measured by $\sup_{0\le t\le T}\mathbb E|X_t-\wt X_t|^{\alpha-1}$. The perturbed drift and Brownian diffusion coefficients are spatially Lipschitz, while the perturbed stable jump coefficient may be H\"older continuous down to the critical exponent $1/\alpha$. The estimate is expressed in terms of the initial error and three law-weighted coefficient errors: the drift error $B$, the Brownian diffusion error $A$, and the stable jump-coefficient error $S$. The Brownian component produces a second-order correction term in It\^o's formula. This term is controlled by a second-derivative estimate for a Komatsu-type mollification of $|x|^{\alpha-1}$. For stable jump coefficients with H\"older exponent $\wt\eta>1/\alpha$, the estimate gives explicit power rates in $B,A,S$; at $\wt\eta=1/\alpha$, it gives a logarithmic rate. The same method yields estimates with predictable forcing errors and time-uniform tail bounds via stopped quasi-martingales.
\end{abstract}



\section{Introduction}
\label{sec:introduction}

Let $\alpha\in(1,2)$ and let $Z$ be a one-dimensional symmetric $\alpha$-stable process. The comparison argument below is based on the one-dimensional Riesz kernel
\[
  \RieszK(x):=|x|^{\alpha-1}.
\]
For the symmetric stable generator $L^\alpha$, this kernel satisfies
\[
  L^\alpha \RieszK=C_\alpha\delta_0
\]
in $\mathcal S'(\mathbb R)$ for a constant $C_\alpha>0$, where $\mathcal S'(\mathbb R)$ denotes the space of tempered distributions and $\delta_0$ is the Dirac mass at the origin. The precise statement and proof are given in Lemma~\ref{lem:stable-generator}. Komatsu's pathwise uniqueness argument uses smooth approximations of $\RieszK$. The same class of approximations is used here to estimate the distance between two solutions driven by the same noise.

The analysis starts from the quantitative stability theory for one-dimensional stable-driven SDEs with non-zero drift developed in Nakagawa~\cite{Nakagawa2026}. In that setting one compares a baseline equation
\[
  X_t=x_0+\int_0^t b(X_s)\dd s+\int_0^t\sigma(X_{s-})\dd Z_s
\]
with a time-inhomogeneous perturbation, and obtains an explicit bound for
\[
  \sup_{0\le t\le T}\mathbb E|X_t-\wt X_t|^{\alpha-1}.
\]
In that result, coefficient errors are measured by expectations along the law of the baseline process. Equivalently, when a transition density is available, they are weighted integral norms. The estimates therefore depend on coefficient differences weighted by the baseline law, rather than on their global supremum over all space; see \cite{Nakagawa2026}.

This paper extends this Komatsu-type $L^{\alpha-1}$ stability method from the stable-only setting to the mixed Brownian--stable system
\begin{align*}
X_t &= x_0+
\int_0^t b(s,X_s)\dd s+
\int_0^t a(s,X_s)\dd W_s+
\int_0^t \sigma(s,X_{s-})\dd Z_s,\\
\wt X_t &= \wt x_0+
\int_0^t \wt b(s,\wt X_s)\dd s+
\int_0^t \wt a(s,\wt X_s)\dd W_s+
\int_0^t \wt\sigma(s,\wt X_{s-})\dd Z_s,
\end{align*}
where the two equations are driven by the same Brownian motion $W$ and the same symmetric stable process $Z$. Throughout the paper, ``pathwise'' refers to this synchronous coupling: the estimate compares two solution paths on the same probability space and not only their marginal laws.

\subsection{Motivation and analytic structure}
\label{subsec:motivation}

The mixed equation contains three coefficient classes. The drift contributes a finite-variation term, the Brownian coefficient contributes a quadratic-variation term, and the stable coefficient contributes a non-local jump-compensator term. The pathwise distance between the two coupled solutions is estimated in terms of the corresponding coefficient discrepancies.

The Brownian term changes the analytic structure of the comparison argument. For the stable-only equation, It\^o's formula applied to a mollified version $u_{\delta,\eps}$ of $|x|^{\alpha-1}$ produces drift and jump-compensator terms. For the mixed equation, the same formula also contains the Brownian second-order correction
\[
  \frac12\int_0^t u_{\delta,\eps}''(X_s-\wt X_s)
  \{a(s,X_s)-\wt a(s,\wt X_s)\}^2\dd s.
\]
Consequently, the proof uses two estimates for the same mollified kernel $u_{\delta,\eps}$: a first-derivative estimate for the drift term and a second-derivative estimate for the Brownian correction. Both estimates retain the dependence on the localization parameters $(\delta,\eps)$, which are optimized in the final rate calculation.

\subsection{Related work}
\label{subsec:related-work}

For Brownian SDEs, the classical Yamada-Watanabe theory gives pathwise uniqueness under a $1/2$-H\"older condition on the diffusion coefficient in one dimension \cite{YamadaWatanabe1971}. For one-dimensional SDEs driven by symmetric $\alpha$-stable processes, Komatsu's method and later refinements use the threshold $1/\alpha$ for the stable coefficient \cite{Komatsu1982,BassBurdzyChen2004,Fournier2013}. The present work considers the mixed regime in which the Brownian diffusion coefficient of the perturbed equation is spatially Lipschitz and the stable coefficient has H\"older exponent in $[1/\alpha,1]$.

The stability result is closest to the stable-only $L^{\alpha-1}$ theory, both without and with drift and time-dependent perturbations \cite{Nakagawa2020,Nakagawa2026}. Quantitative stability rates for one-dimensional SDEs have also been studied by Hashimoto and Tsuchiya through Yamada-Watanabe type arguments, including discontinuous-coefficient settings and stable-noise cases \cite{HashimotoTsuchiya2014}. Hashimoto studied approximation and stability for SDEs driven by a symmetric $\alpha$-stable process under non-Lipschitz coefficient conditions \cite{Hashimoto2013}. Relative to these works, the present analysis treats Brownian and stable noises in one coupled equation, allows time-dependent baseline and perturbed coefficients, with perturbed spatial moduli in $L^1$, $L^2$, and $L^\alpha$, measures the three coefficient errors by $B,A,S$, and proves both non-critical power rates and the critical logarithmic rate. The additional analytic term is the Brownian second-order correction, which is absent in the pure stable equation. Related background on stable-driven equations and jump SDE approximation can be found in Bass's survey and in works such as Priola and Qiao \cite{Bass2002,Priola2012,Qiao2014}.

Table~\ref{tab:comparison} records the comparison with the closest one-dimensional stability results. The entries list the features used in the present argument: synchronous coupling by Brownian and stable noises, simultaneous errors in the drift, Brownian diffusion, and stable jump coefficients, and the critical logarithmic estimate.
\begin{table}[h]
\centering
\small
\begin{tabular}{>{\raggedright\arraybackslash}p{0.23\textwidth}
          >{\raggedright\arraybackslash}p{0.19\textwidth}
          >{\raggedright\arraybackslash}p{0.27\textwidth}
          >{\raggedright\arraybackslash}p{0.20\textwidth}}
\hline
Work & Driving noise & Coefficient perturbations & Distance and rate features \\
\hline
Hashimoto--Tsuchiya~\cite{HashimotoTsuchiya2014} & Brownian motion or a symmetric $\alpha$-stable process (in separate settings) & one-dimensional stability problems with non-Lipschitz or discontinuous coefficients & explicit convergence rates under the conditions considered \\
Nakagawa~\cite{Nakagawa2026} & symmetric stable process & drift and stable jump coefficient errors & $L^{\alpha-1}$ pathwise rate with law-weighted $B,S$ and a critical logarithmic rate \\
This work & common Brownian motion and common symmetric stable process & drift error $B$, Brownian diffusion error $A$, stable jump coefficient error $S$ & law-weighted mixed-noise estimate, second-order Komatsu bound, power and critical logarithmic rates \\
\hline
\end{tabular}
\caption{Main features used for comparison with closely related one-dimensional stability results.}
\label{tab:comparison}
\end{table}

The difference from Nakagawa~\cite{Nakagawa2026} appears at three mathematical points. First, the Brownian component produces a second-order correction term, so the Komatsu mollifier has to be estimated at the level of $u_{\delta,\eps}''$ as well as $u_{\delta,\eps}'$. Second, the law-weighted distance contains the quadratic Brownian error $A$, and its admissible growth is governed by the stable tails of the baseline law. Third, the optimization of the fundamental estimate balances the mixed error scales $B,A,S$ simultaneously, including the additional $\delta^2 A^2$ term in the critical logarithmic case. The proof is carried out directly for the mixed equation.

This analysis develops the stable-only law-weighted framework by adding the Brownian diffusion coefficient error
\[
  A^2=\int_0^T\mathbb E |a(s,X_s)-\wt a(s,X_s)|^2\dd s
    =\int_0^T\int_{\mathbb R}|a(s,y)-\wt a(s,y)|^2\,\mu_s^X(\md y)\dd s,
\]
where $\mu_s^X=\mathcal L(X_s)$. If this law has a density, the last expression becomes the corresponding density-weighted spatial norm. The pathwise argument itself does not require a transition density. For additive mixed Brownian--stable processes, heat kernel estimates for operators of the form $\Delta+a^\alpha\Delta^{\alpha/2}$ are known; Chen, Kim and Song identify the associated process as the sum of Brownian motion and an independent symmetric $\alpha$-stable process and prove sharp two-sided estimates in suitable domains \cite{ChenKimSong2010}. Gradient perturbations of such mixed generators have also been studied by Chen and Hu \cite{ChenHu2015}. These heat-kernel results are relevant for converting law-weighted distances into spatial integral norms. The stability proof below is formulated directly in terms of the corresponding law-weighted expectations.

The law-weighted distances encode the tail behavior of the baseline law. In particular, $A^2$ contains the second moment of the Brownian coefficient difference along the baseline law; for a stable baseline law, this may be infinite for linearly growing differences. Section~\ref{subsec:finite-distances} gives sufficient growth conditions under which $B,A,S$ are finite, while uniformly bounded coefficient differences yield the law-independent bounds recorded in Section~\ref{sec:consequences}.

\subsection{Summary of results}
\label{subsec:contributions}

The results used in the main proofs are summarized as follows.

First, we formulate a fixed-coupling pathwise stability estimate for a pair of solutions driven by the same Brownian motion and the same symmetric stable process. The coefficient errors are measured by the three law-weighted quantities
\[
  B=\int_0^T\mathbb E |b(s,X_s)-\wt b(s,X_s)|\dd s,
\]
\[
  A^2=\int_0^T\mathbb E |a(s,X_s)-\wt a(s,X_s)|^2\dd s,
\]
\[
  S^\alpha=\int_0^T\mathbb E |\sigma(s,X_s)-\wt\sigma(s,X_s)|^\alpha\dd s.
\]
The quantity absent from the stable-only framework is $A$, which measures the Brownian diffusion coefficient error. The finite-distance criteria below identify perturbation classes for which the law-weighted quantities $B,A,S$ are finite. The Brownian error is quadratic, so its admissible spatial growth is tied to the stable moments of the baseline process. These criteria are recorded together with model examples: one with three non-vanishing coefficient distances, one with a bounded Brownian perturbation, and one with a sublinear spatial Brownian perturbation for which the condition $2\theta<\alpha$ appears explicitly.

Second, we prove a localized fundamental estimate. For $\delta\ge2$ and $\eps\in(0,1)$,
\begin{align*}
\sup_{0\le t\le T}\mathbb E|X_t-\wt X_t|^{\alpha-1}
\le C\Bigg[&|x_0-\wt x_0|^{\alpha-1}
+\eps^{\alpha-1}
+\frac{\delta}{\log\delta}\frac{B}{\eps^{2-\alpha}}\\
&+\frac{\delta^2}{\log\delta}\frac{A^2}{\eps^{3-\alpha}}
+\frac{\delta}{\log\delta}\frac{S^\alpha}{\eps}
+\frac{\delta^2}{\log\delta}\eps^{\alpha-1}
+\frac{\eps^{\alpha\wt\eta-1}}{\log\delta}
\Bigg].
\end{align*}
The term $A^2\eps^{-(3-\alpha)}$ is generated by the Brownian diffusion coefficient error. Its derivation uses a second-derivative estimate for the Komatsu-type mollified kernel.

Third, optimizing the localization parameters gives explicit rates. In the non-critical case $\wt\eta\in(1/\alpha,1]$ this yields power rates in $B,A,S$, while at the critical exponent $\wt\eta=1/\alpha$ it yields a logarithmic rate; the zero-error case is understood by taking limits. The same localized estimates also yield a time-uniform convergence-in-probability bound through the classical maximal inequality for stopped quasi-martingales. Constant-coefficient examples show that the exponent $\alpha-1$ is attained in the Lipschitz subclass.

The localized estimate also admits predictable forcing errors. This formulation separates the comparison mechanism from the special form of a coefficient perturbation. Predictable drift, Brownian-diffusion, and stable-jump errors are absorbed into the same three scales $B,A,S$; subsequent discretization or model-error analyses may use this form after the corresponding one-step or model-error bounds have been established.

\subsection{Organization}
\label{subsec:organization}

The paper is organized as follows. \Cref{sec:model} introduces the coupled equations and the standing assumptions. \Cref{sec:distances} defines the law-weighted coefficient distances, gives verifiable sufficient conditions for their finiteness, and records examples with H\"older non-Lipschitz stable coefficients in the non-critical range. \Cref{sec:main-results} states the localized fundamental estimate and the optimized non-critical and critical rates. \Cref{sec:consistency} records consistency checks in constant-coefficient examples. \Cref{sec:auxiliary} constructs the log-annular mollifier, proves the first- and second-derivative estimates, and establishes the Riesz-kernel identity used for the stable generator. \Cref{sec:localization} supplies the localization and a priori estimates needed before Gronwall's inequality is applied. \Cref{sec:proof-fundamental} proves the fundamental estimate by separately treating the drift, Brownian correction, and stable jump terms. \Cref{sec:probability} proves the convergence-in-probability estimate using stopped quasi-martingales. \Cref{sec:forcing} states the predictable-forcing version of the fundamental estimate. Finally, \Cref{sec:consequences,sec:concluding} present uniform coefficient bounds, the time-homogeneous approximation corollary, and concluding remarks.

\section{Model and assumptions}
\label{sec:model}

For real numbers $x$ and $y$, write
\[
  x\wedge y:=\min\{x,y\},\qquad x\vee y:=\max\{x,y\}.
\]
The symbol $\1_A$ denotes the indicator of a set $A$, $\mathcal L(\xi)$ denotes the law of a random variable $\xi$, and $\|\cdot\|_\infty$ denotes the uniform norm. Convolution on $\R$ is written as
\[
  (f*g)(x):=\int_\R f(x-y)g(y)\dd y
\]
whenever the integral is well defined. Time-homogeneous coefficients are identified with their constant-in-time extensions to $[0,T]\times\R$.

Let $(\Omega,\F,(\F_t)_{0\le t\le T},\Pbb)$ be a filtered probability space satisfying the usual conditions. Let $W=(W_t)_{0\le t\le T}$ be a one-dimensional Brownian motion. Let $Z=(Z_t)_{0\le t\le T}$ be a one-dimensional symmetric $\alpha$-stable process, independent of $W$, with L\'evy measure (see, e.g., \cite{Applebaum2009})
\[
  \nu(\md z)=c_\alpha |z|^{-1-\alpha}\,\md z.
\]
Let $N$ and $\wt N$ be the Poisson random measure and compensated Poisson random measure associated with $Z$. The symmetric compensation convention is used for stochastic integration with respect to $Z$. More explicitly, for predictable integrands $\gamma$ for which the integral is defined,
\[
  \int_0^t \gamma_s\dd Z_s
  =\int_0^t\int_{\R\setminus\{0\}}\gamma_s z\,\wt N(\md s,\md z),
\]
where the drift contribution from symmetric annuli vanishes by symmetry. Indeed, for every predictable $\gamma_s$ and every $r>0$, the compensator of the annular first-order term is proportional to
\[
  \gamma_s\int_{|z|>r}z\,\nu(\md z)=0,
\]
whenever the annular integral is absolutely convergent; this is the case for the annuli used below because $\alpha>1$.

Consider
\begin{align}
X_t
&=x_0+\int_0^t b(s,X_s)\dd s+\int_0^t a(s,X_s)\dd W_s
 +\int_0^t \sigma(s,X_{s-})\dd Z_s, \label{eq:X}\\
\wt X_t
&=\wt x_0+\int_0^t \wt b(s,\wt X_s)\dd s+\int_0^t \wt a(s,\wt X_s)\dd W_s
 +\int_0^t \wt\sigma(s,\wt X_{s-})\dd Z_s. \label{eq:Xtilde}
\end{align}
Set
\[
  Y_t:=X_t-\wt X_t.
\]

\begin{assumption}[Baseline coefficients]\label{ass:baseline}
The functions $b,a,\sigma:[0,T]\times\R\to\R$ are jointly Borel measurable and have at most linear growth uniformly in time; that is, there exists $K>0$ such that
\[
  |b(t,x)|+|a(t,x)|+|\sigma(t,x)|\le K(1+|x|),
  \qquad (t,x)\in[0,T]\times\R.
\]
\end{assumption}

\begin{assumption}[Perturbed coefficients]\label{ass:perturbed}
The functions $\wt b,\wt a,\wt\sigma:[0,T]\times\R\to\R$ are jointly measurable. There exist nonnegative functions
\[
  \ell_b\in L^1(0,T),\qquad \ell_a\in L^2(0,T),\qquad \ell_\sigma\in L^\alpha(0,T),
\]
and an exponent
\[
  \wt\eta\in\left[\frac1\alpha,1\right]
\]
such that, for all $t\in[0,T]$ and $x,y\in\R$,
\begin{align}
  |\wt b(t,x)-\wt b(t,y)|&\le \ell_b(t)|x-y|, \label{eq:b-lip}\\
  |\wt a(t,x)-\wt a(t,y)|&\le \ell_a(t)|x-y|, \label{eq:a-lip}\\
  |\wt\sigma(t,x)-\wt\sigma(t,y)|&\le \ell_\sigma(t)\bigl(|x-y|^{\wt\eta}\vee |x-y|\bigr). \label{eq:sigma-holder}
\end{align}
The endpoint $\wt\eta=1/\alpha$ is the critical case treated in Theorem~\ref{thm:critical}, while $\wt\eta=1$ is the Lipschitz case included in Theorem~\ref{thm:noncritical}.
\end{assumption}

\begin{assumption}[Fixed coupled solutions and finite coefficient distances]
\label{ass:solutions}
On the filtered probability space above, let $X$ and $\wt X$ be c\`adl\`ag adapted solutions of \eqref{eq:X} and \eqref{eq:Xtilde}, respectively, driven by the same pair $(W,Z)$. The stochastic integrals in both equations are assumed to be well defined. The coefficient distances are computed along this fixed baseline solution $X$, and the quantities $B,A,S$ defined in \Cref{sec:distances} are assumed finite.
\end{assumption}

\begin{remark}[Fixed-coupling formulation]
The estimates are formulated for a prescribed synchronous coupling. The existence of the coupled pair is part of the standing framework and is logically separate from the comparison estimate. Whenever a well-posedness result or an explicit construction supplies such a pair, the results below apply to that pair.
\end{remark}

\begin{remark}[Asymmetric structure]
The stability estimate is asymmetric because the coefficient distances are measured under the law of the baseline process $X$. It therefore quantifies the sensitivity of this fixed baseline law to a perturbation realized on the same probability space.
\end{remark}

\begin{remark}[Use of the local H\"older modulus]
The proof uses the H\"older part of \eqref{eq:sigma-holder} only on the support of $\psi_{\delta,\eps}(Y_s)$, where $|Y_s|<\eps\le1$. On this region $|Y_s|^{\wt\eta}\vee |Y_s|=|Y_s|^{\wt\eta}$. For large differences the linear part in \eqref{eq:sigma-holder} is used only in stopped a priori estimates.
\end{remark}

\begin{remark}[Constants and localization]
The proof uses localization rather than global deterministic envelopes of the perturbed coefficients. The martingale and Gronwall arguments are first performed up to stopping times. On stopped intervals, the coefficients are controlled by their direct differences along $X$ and by the moduli $\ell_b,\ell_a,\ell_\sigma$. The constants in the fundamental estimate depend on $T,\alpha,\wt\eta$, the norms of these moduli, and the fixed mollifier profile $\varphi$. The linear-growth constant $K$ is used in the moment estimate of Proposition~\ref{prop:baseline-moments}, but not in the comparison estimate once the coupled pair and the finite distances $B,A,S$ are fixed.
\end{remark}

\begin{remark}[Stopped local integrability]
On stopped intervals, the local integrability required in the comparison argument follows from the direct coefficient differences and the spatial moduli. In particular,
\begin{align*}
|\wt b(s,\wt X_s)|&\le |b(s,X_s)|+|b(s,X_s)-\wt b(s,X_s)|+\ell_b(s)|Y_s|,\\
|\wt a(s,\wt X_s)|&\le |a(s,X_s)|+|a(s,X_s)-\wt a(s,X_s)|+\ell_a(s)|Y_s|,\\
|\wt\sigma(s,\wt X_s)|&\le |\sigma(s,X_s)|+|\sigma(s,X_s)-\wt\sigma(s,X_s)|
 +\ell_\sigma(s)(|Y_s|^{\wt\eta}\vee |Y_s|).
\end{align*}
The drift requires only $L^1$-integrability, which is provided by $B<\infty$, $\ell_b\in L^1(0,T)$, and the stopping. The Brownian and stable terms are handled similarly in the $L^2$ and $L^\alpha$ scales, respectively.
\end{remark}

\section{Law-weighted coefficient distances}
\label{sec:distances}

For each $s\in[0,T]$, let
\[
  \mu_s^X:=\mathcal L(X_s)
\]
denote the marginal law of the baseline process. The coefficient errors are measured directly against these laws. Define
\begin{align}
B&:=\int_0^T\E |b(s,X_s)-\wt b(s,X_s)|\dd s
  =\int_0^T\int_\R |b(s,y)-\wt b(s,y)|\,\mu_s^X(\md y)\dd s, \label{eq:B-exp}\\
A^2&:=\int_0^T\E |a(s,X_s)-\wt a(s,X_s)|^2\dd s
  =\int_0^T\int_\R |a(s,y)-\wt a(s,y)|^2\,\mu_s^X(\md y)\dd s, \label{eq:A-exp}\\
S^\alpha&:=\int_0^T\E |\sigma(s,X_s)-\wt\sigma(s,X_s)|^\alpha\dd s
  =\int_0^T\int_\R |\sigma(s,y)-\wt\sigma(s,y)|^\alpha\,\mu_s^X(\md y)\dd s. \label{eq:S-exp}
\end{align}
These definitions do not require a transition density. They also show directly that coefficient differences in regions receiving little mass under $\mu_s^X$ have correspondingly little influence on the estimates.

If $\mu_s^X$ is absolutely continuous with respect to Lebesgue measure, write
\[
  \mu_s^X(\md y)=p_s^X(y)\,\md y.
\]
Then the same quantities admit the spatial representations
\begin{align}
B&=\int_0^T\int_\R |b(s,y)-\wt b(s,y)|p_s^X(y)\dd y\dd s, \label{eq:B-density}\\
A^2&=\int_0^T\int_\R |a(s,y)-\wt a(s,y)|^2p_s^X(y)\dd y\dd s, \label{eq:A-density}\\
S^\alpha&=\int_0^T\int_\R |\sigma(s,y)-\wt\sigma(s,y)|^\alpha p_s^X(y)\dd y\dd s. \label{eq:S-density}
\end{align}
When the chosen baseline solution is Markov and admits a transition density, $p_s^X(y)$ is the transition density from the initial state $x_0$ at time $0$. The density representation is optional and is not used in the proof of the fundamental estimate. For example, in the time-homogeneous stable-only setting, transition-density results are available under boundedness, non-degeneracy, and H\"older-type conditions on the stable coefficient; see Knopova and Kulik~\cite{KnopovaKulik2018}. Other coefficient classes require the corresponding external density theory.

\subsection{Verifiable sufficient conditions for finite distances}
\label{subsec:finite-distances}

The finite-distance conditions for $A$ and $S$ reflect the tail behavior of stable laws. The following moment bound and corollary give a concrete class of coefficient differences for which the distances in \eqref{eq:B-exp}-\eqref{eq:S-exp} are finite.

\begin{proposition}[Moment bounds for the baseline equation]\label{prop:baseline-moments}
Assume Assumption~\ref{ass:baseline}, and let $X$ be any c\`adl\`ag adapted solution of \eqref{eq:X} for which the stochastic integrals are well defined. Then, for every $q\in(0,\alpha)$, there exists a constant $C=C(q,\alpha,K,T,c_\alpha)>0$ such that
\[
  \sup_{0\le t\le T}\E |X_t|^q\le C(1+|x_0|^q).
\]
\end{proposition}

\begin{proof}
Fix $q\in(0,\alpha)$ and set
\[
  V(x):=(1+x^2)^{q/2}.
\]
Then $V\in C^2(\R)$, and there exist constants $0<c<C$ such that
\[
  c(1+|x|^q)\le V(x)\le C(1+|x|^q),
\]
and
\begin{equation}\label{eq:V-derivatives}
  |V'(x)|\le C(1+|x|)^{q-1},\qquad
  |V''(x)|\le C(1+|x|)^{q-2}.
\end{equation}
Since $q<\alpha<2$, the exponent $q-2$ is negative. If
\[
  |\xi-x|\le \rho(x):=\frac12(1+|x|),
\]
then $1+|\xi|\ge(1+|x|)/2$, and hence
\begin{equation}\label{eq:V-local-second}
  |V''(\xi)|\le C(1+|x|)^{q-2}.
\end{equation}

Let
\[
  \vartheta_N:=\inf\{t\in[0,T]:|X_t|\ge N\}\wedge T.
\]
For $(s,x)\in[0,T]\times\R$, define the image L\'evy measure
\[
  \nu_{s,x}(\md w):=c_\alpha|\sigma(s,x)|^\alpha |w|^{-1-\alpha}\,\md w
\]
and the jump part of the generator by
\begin{equation}\label{eq:Jsigma-definition}
J_\sigma V(s,x):=
\int_{\R\setminus\{0\}}
\bigl[V(x+w)-V(x)-wV'(x)\1_{\{|w|\le\rho(x)\}}\bigr]\nu_{s,x}(\md w).
\end{equation}
The replacement of the cutoff $1$ by $\rho(x)$ does not change the value: the difference is the integral of the odd function $wV'(x)$ over a symmetric annulus, which is separated from the origin and therefore absolutely integrable.

Applying It\^o's formula for jump processes to $V(X_{t\wedge\vartheta_N})$ gives a stopped local martingale plus the finite-variation part
\[
V(x_0)+\int_0^{t\wedge\vartheta_N}
\left[b(s,X_s)V'(X_s)+\frac12a(s,X_s)^2V''(X_s)+J_\sigma V(s,X_s)\right]\dd s.
\]
On $[0,\vartheta_N]$, the predictable coefficients are evaluated at $X_{s-}$, which is bounded by $N$ before the exit jump; hence the Brownian martingale is square-integrable. For the jump martingale, set
\begin{equation}\label{eq:FV-definition}
F_s^V(z):=V\bigl(X_{s-}+\sigma(s,X_{s-})z\bigr)-V(X_{s-}).
\end{equation}
On $[0,\vartheta_N]$, one has $|F_s^V(z)|\le C_N|z|$ for $|z|\le1$ and $|F_s^V(z)|\le C_N(1+|z|^q)$ for $|z|>1$. Since $q<\alpha$,
\[
\E\int_0^{T\wedge\vartheta_N}\int_{\R\setminus\{0\}}
\bigl(|F_s^V(z)|^2\wedge |F_s^V(z)|\bigr)\nu(\md z)\dd s<\infty.
\]
Thus both stopped stochastic integrals are true martingales and have expectation zero. Consequently,
\begin{equation}\label{eq:moment-ito-stopped}
\E V(X_{t\wedge\vartheta_N})
=V(x_0)+\E\int_0^{t\wedge\vartheta_N}
\left[b(s,X_s)V'(X_s)+\frac12a(s,X_s)^2V''(X_s)+J_\sigma V(s,X_s)\right]\dd s.
\end{equation}

The drift and Brownian terms are controlled by the uniform linear-growth assumption and \eqref{eq:V-derivatives}:
\[
|b(s,x)V'(x)|+\frac12|a(s,x)|^2|V''(x)|
\le C(1+|x|)^q\le C(1+V(x)).
\]
For the jump part, the small-jump region $|w|\le\rho(x)$ and Taylor's theorem give, by \eqref{eq:V-local-second},
\[
  |V(x+w)-V(x)-wV'(x)|
  \le C(1+|x|)^{q-2}w^2.
\]
Hence, uniformly in $s$,
\begin{align*}
\int_{|w|\le\rho(x)}&|V(x+w)-V(x)-wV'(x)|\nu_{s,x}(\md w)\\
&\le C(1+|x|)^{q-2}|\sigma(s,x)|^\alpha
    \int_{|w|\le\rho(x)}|w|^{1-\alpha}\dd w\\
&\le C(1+|x|)^{q-2}(1+|x|)^\alpha(1+|x|)^{2-\alpha}
 =C(1+|x|)^q.
\end{align*}
For $|w|>\rho(x)$, use
\[
  |V(x+w)-V(x)|\le C\{(1+|x|)^q+|w|^q\}.
\]
Since $q<\alpha$,
\begin{align*}
|\sigma(s,x)|^\alpha\int_{|w|>\rho(x)}|w|^q|w|^{-1-\alpha}\dd w
&\le C(1+|x|)^\alpha\rho(x)^{q-\alpha}
 \le C(1+|x|)^q,\\
(1+|x|)^q|\sigma(s,x)|^\alpha\int_{|w|>\rho(x)}|w|^{-1-\alpha}\dd w
&\le C(1+|x|)^q.
\end{align*}
Therefore
\[
 b(s,x)V'(x)+\frac12a(s,x)^2V''(x)+J_\sigma V(s,x)
 \le C(1+V(x))
\]
uniformly in $s$. Taking expectations in \eqref{eq:moment-ito-stopped}, using
$\1_{\{s\le\vartheta_N\}}(1+V(X_s))\le1+V(X_{s\wedge\vartheta_N})$, and applying Gronwall's inequality yields
\[
  \sup_{0\le t\le T}\E V(X_{t\wedge\vartheta_N})\le C(1+V(x_0)),
\]
with $C$ independent of $N$. Since a c\`adl\`ag path is bounded on $[0,T]$, one has $\vartheta_N=T$ for all sufficiently large $N$, almost surely. Fatou's lemma gives the result.
\end{proof}

\begin{corollary}[Admissible coefficient errors]\label{cor:admissible-distances}
Assume Assumptions~\ref{ass:baseline} and~\ref{ass:perturbed}, and let $X$ solve \eqref{eq:X}. Suppose that there exist measurable functions
\[
  h_b\in L^1(0,T),\qquad h_a\in L^2(0,T),\qquad h_\sigma\in L^\alpha(0,T)
\]
and exponents
\[
  \theta_b\in[0,1],\qquad \theta_a\in[0,\alpha/2),\qquad \theta_\sigma\in[0,1)
\]
such that, for all $(s,y)$,
\begin{align*}
|b(s,y)-\wt b(s,y)|&\le h_b(s)(1+|y|^{\theta_b}),\\
|a(s,y)-\wt a(s,y)|&\le h_a(s)(1+|y|^{\theta_a}),\\
|\sigma(s,y)-\wt\sigma(s,y)|&\le h_\sigma(s)(1+|y|^{\theta_\sigma}).
\end{align*}
Then $B,A,S<\infty$. More precisely, if $m_r:=\sup_{0\le t\le T}\E|X_t|^r$ for $r<\alpha$, then
\[
B\le \|h_b\|_{L^1(0,T)}(1+m_{\theta_b}),\quad
A^2\le 2\|h_a\|_{L^2(0,T)}^2(1+m_{2\theta_a}),\quad
S^\alpha\le 2^{\alpha-1}\|h_\sigma\|_{L^\alpha(0,T)}^\alpha(1+m_{\alpha\theta_\sigma}).
\]
\end{corollary}

\begin{proof}
Use $(1+r)^2\le2(1+r^2)$ and $(1+r)^\alpha\le2^{\alpha-1}(1+r^\alpha)$ for $r\ge0$, and then apply Proposition~\ref{prop:baseline-moments}. The required moments are finite because $\theta_b\le1<\alpha$, $2\theta_a<\alpha$, and $\alpha\theta_\sigma<\alpha$.
\end{proof}

\begin{remark}[Finite-distance criteria under stable tails]
The growth thresholds in Corollary~\ref{cor:admissible-distances} reflect the heavy-tail regime. If the baseline is $X_t=Z_t$ and $\wt a(s,y)=c y$ with $c\ne0$, then
\[
  A^2=c^2\int_0^T\E|Z_s|^2\dd s=\infty,
\]
because a symmetric $\alpha$-stable random variable has no moment of order $q\ge\alpha$. Similarly, a linear perturbation of the stable coefficient gives $S=\infty$. This example identifies the tail behavior encoded in the law-weighted distances. For bounded coefficient differences, the uniform bounds in Section~\ref{sec:consequences} give a directly verifiable sufficient condition.
\end{remark}

\begin{remark}[Coupled pairs in the examples]
The first example below invokes standard one-dimensional stable-SDE well-posedness results to construct a coupled pair. The subsequent mixed examples compute the distances $B,A,S$ for any synchronous coupled pair satisfying Assumption~\ref{ass:solutions}. This keeps the coefficient calculations separate from the existence theory for the full mixed equation.
\end{remark}

\begin{proposition}[A non-Lipschitz example in the non-critical regime]\label{prop:nonlipschitz-example}
Let $\wt\eta\in(1/\alpha,1)$ and fix $0<\lambda<1\wedge T^{-1/\alpha}$. Consider the time-homogeneous coefficients
\[
  b(t,x)=\wt b(t,x)\equiv0,
  \qquad
  a(t,x)=\wt a(t,x)\equiv0,
  \qquad
  \sigma(t,x)\equiv1,
  \qquad
  \wt\sigma(t,x)=1+\lambda\bigl(|x|^{\wt\eta}\wedge1\bigr).
\]
Then $\wt\sigma$ is bounded, bounded away from zero, and fails to be Lipschitz at the origin, while satisfying Assumption~\ref{ass:perturbed} with exponent $\wt\eta$. Moreover, there is a coupled pair of solutions driven by the same stable process $Z$ for which Assumption~\ref{ass:solutions} holds and
\[
  B=A=0,
  \qquad
  S^\alpha\le \lambda^\alpha T<\infty.
\]
Consequently, $S<1$, and the non-critical estimate of Theorem~\ref{thm:noncritical} applies to a non-Lipschitz stable coefficient. In particular, if $x_0=\wt x_0$, then
\[
  \sup_{0\le t\le T}\E|X_t-\wt X_t|^{\alpha-1}
  \le C\lambda^{\alpha-1/\wt\eta},
\]
where $C$ depends on $T,\alpha,\wt\eta$ and the fixed coefficient bounds, but not on $\lambda\in(0,1\wedge T^{-1/\alpha})$.
\end{proposition}

\begin{proof}
The function $x\mapsto |x|^{\wt\eta}\wedge1$ is $\wt\eta$-H\"older on $\R$ and fails to be Lipschitz at the origin, since $|h|^{\wt\eta}/|h|\to\infty$ as $h\to0$. Thus $1\le \wt\sigma\le1+\lambda$ and $\wt\sigma$ satisfies \eqref{eq:sigma-holder} with a constant modulus. The baseline solution is explicit, $X_t=x_0+Z_t$. For the perturbed equation
\[
  \wt X_t=\wt x_0+\int_0^t \wt\sigma(s,\wt X_{s-})\dd Z_s,
\]
the existence and pathwise uniqueness results for one-dimensional stable SDEs with bounded non-degenerate H\"older coefficients of exponent greater than $1/\alpha$ apply; see Komatsu~\cite{Komatsu1982} and Fournier~\cite{Fournier2013}. Thus the perturbed equation has a strong solution driven by the prescribed $Z$, and the two equations can be realized on the same probability space with the same driving process. Finally, $b=\wt b$ and $a=\wt a$, so $B=A=0$, while
\[
  |\sigma(s,y)-\wt\sigma(s,y)|\le \lambda
\]
for all $(s,y)\in[0,T]\times\R$, and therefore $S^\alpha\le \lambda^\alpha T$. If $x_0=\wt x_0$, Theorem~\ref{thm:noncritical} yields
\[
  \sup_{0\le t\le T}\E|X_t-\wt X_t|^{\alpha-1}
  \le C S^{\alpha-1/\wt\eta}
  \le C T^{(\alpha-1/\wt\eta)/\alpha}\lambda^{\alpha-1/\wt\eta},
\]
and the factor depending on $T$ is absorbed into $C$.

\end{proof}

\begin{proposition}[A bounded Brownian perturbation in a mixed equation]\label{prop:mixed-brownian-error}
Let $\wt\eta\in(1/\alpha,1)$, let $0<\lambda<1$, let $\kappa\ne0$, and let $r\in L^2(0,T)$ be deterministic. Consider
\[
  b(t,x)=\wt b(t,x)\equiv0,
  \qquad
  a(t,x)\equiv\kappa,
  \qquad
  \wt a(t,x)=\kappa+r(t),
\]
and
\[
  \sigma(t,x)=\wt\sigma(t,x)=1+\lambda\bigl(|x|^{\wt\eta}\wedge1\bigr).
\]
Then the baseline equation contains both the Brownian and stable terms, the perturbed equation is driven by the same two noises, and the stable coefficient is bounded, bounded away from zero, and non-Lipschitz at the origin. For any coupled solution pair satisfying Assumption~\ref{ass:solutions}, the coefficient distances are
\[
  B=0,
  \qquad
  A^2=\int_0^T |r(t)|^2\dd t,
  \qquad
  S=0.
\]
Consequently, if $x_0=\wt x_0$ and $A=\|r\|_{L^2(0,T)}<1$, Theorem~\ref{thm:noncritical} gives
\[
  \sup_{0\le t\le T}\E|X_t-\wt X_t|^{\alpha-1}
  \le C\|r\|_{L^2(0,T)}^{\frac{2(\alpha\wt\eta-1)}{\alpha\wt\eta-\alpha+2}}.
\]
Thus, in this mixed Brownian--stable equation, the estimate reduces to the Brownian diffusion error, while the stable coefficient is H\"older and fails to be Lipschitz.
\end{proposition}

\begin{proof}
The coefficient $x\mapsto1+\lambda(|x|^{\wt\eta}\wedge1)$ has the same H\"older and non-Lipschitz properties as in Proposition~\ref{prop:nonlipschitz-example}. The Brownian coefficients are spatially constant, hence Lipschitz with modulus zero. Since the drift coefficients coincide, $B=0$; the Brownian coefficient difference equals $r(t)$, so
\[
  A^2=\int_0^T|r(t)|^2\dd t,
\]
and since the stable coefficients coincide, $S=0$. The displayed estimate is the $A$-term in Theorem~\ref{thm:noncritical}.
\end{proof}

\begin{corollary}[Time-uniform probability estimate for the mixed example]\label{cor:mixed-example-probability}
In the setting of Proposition~\ref{prop:mixed-brownian-error}, assume $x_0=\wt x_0$ and $\|r\|_{L^2(0,T)}<1$. Then, for every $h>0$,
\[
  \mathbb P\left(\sup_{0\le t\le T}|X_t-\wt X_t|>h\right)
  \le
  \frac{C}{h^{\alpha-1}}
  \|r\|_{L^2(0,T)}^{\frac{2(\alpha\wt\eta-1)}{\alpha\wt\eta-\alpha+2}}.
\]
Equivalently,
\[
  \mathbb P\left(\sup_{0\le t\le T}|X_t-\wt X_t|^{\alpha-1}>h\right)
  \le
  \frac{C}{h}
  \|r\|_{L^2(0,T)}^{\frac{2(\alpha\wt\eta-1)}{\alpha\wt\eta-\alpha+2}}.
\]
The constant is independent of $h$ and $r$; for $0<\lambda<1$ it can be chosen uniformly in $\lambda$.
\end{corollary}

\begin{proof}
This is Theorem~\ref{thm:probability} applied to Proposition~\ref{prop:mixed-brownian-error}, where $B=0$, $A=\|r\|_{L^2(0,T)}$, and $S=0$. The first display follows from the second by replacing $h$ with $h^{\alpha-1}$.
\end{proof}

\begin{proposition}[A mixed example with all three coefficient distances]\label{prop:all-errors-example}
Let $\wt\eta\in(1/\alpha,1)$, let $0<\lambda<1$, let $\varpi>0$, and let $b_0,a_0\in\R$. Let $\beta\in L^1(0,T)$ and $r\in L^2(0,T)$ be deterministic. Consider
\[
  b(t,x)\equiv b_0,
  \qquad
  \wt b(t,x)=b_0+\beta(t),
\]
\[
  a(t,x)\equiv a_0,
  \qquad
  \wt a(t,x)=a_0+r(t),
\]
and
\[
  \sigma(t,x)=1+\lambda\bigl(|x|^{\wt\eta}\wedge1\bigr),
  \qquad
  \wt\sigma(t,x)=1+\varpi+\lambda\bigl(|x|^{\wt\eta}\wedge1\bigr).
\]
Then both stable coefficients are bounded, bounded away from zero, and H\"older while failing to be Lipschitz at the origin. For any coupled solution pair satisfying Assumption~\ref{ass:solutions}, the coefficient distances satisfy
\[
  B=\|\beta\|_{L^1(0,T)},
  \qquad
  A=\|r\|_{L^2(0,T)},
  \qquad
  S=\varpi T^{1/\alpha}.
\]
Consequently, if $x_0=\wt x_0$ and these three quantities are smaller than one, then Theorem~\ref{thm:noncritical} gives
\[
\sup_{0\le t\le T}\E|X_t-\wt X_t|^{\alpha-1}
\le C\max\left\{
\|\beta\|_{L^1(0,T)}^{\frac{\alpha\wt\eta-1}{\alpha\wt\eta-\alpha+1}},
\|r\|_{L^2(0,T)}^{\frac{2(\alpha\wt\eta-1)}{\alpha\wt\eta-\alpha+2}},
(\varpi T^{1/\alpha})^{\alpha-1/\wt\eta}
\right\}.
\]
\end{proposition}

\begin{proof}
The H\"older and non-Lipschitz assertions follow as in Proposition~\ref{prop:nonlipschitz-example}. The drift and Brownian coefficients are spatially constant, so their moduli vanish. Along the baseline process,
\[
  |b(t,X_t)-\wt b(t,X_t)|=|\beta(t)|,
  \qquad
  |a(t,X_t)-\wt a(t,X_t)|=|r(t)|,
\]
and
\[
  |\sigma(t,X_t)-\wt\sigma(t,X_t)|=\varpi.
\]
Thus $B=\|\beta\|_{L^1(0,T)}$, $A=\|r\|_{L^2(0,T)}$, and $S=\varpi T^{1/\alpha}$. The estimate follows from Theorem~\ref{thm:noncritical}.
\end{proof}

\begin{proposition}[A law-weighted Brownian error with sublinear spatial growth]\label{prop:mixed-tail-brownian-error}
Let $\wt\eta\in(1/\alpha,1)$, $0<\lambda<1$, $\kappa\ne0$, and let $0<\theta<\alpha/2$. Put
\[
  \phi_\theta(x):=(1+x^2)^{\theta/2},\qquad x\in\R,
\]
and let $r\in L^2(0,T)$ be deterministic. Consider
\[
  b(t,x)=\wt b(t,x)\equiv0,
  \qquad
  a(t,x)\equiv\kappa,
  \qquad
  \wt a(t,x)=\kappa+r(t)\phi_\theta(x),
\]
and
\[
  \sigma(t,x)=\wt\sigma(t,x)=1+\lambda\bigl(|x|^{\wt\eta}\wedge1\bigr).
\]
Then the stable coefficient is bounded, bounded away from zero, and H\"older while failing to be Lipschitz at the origin. The perturbed Brownian coefficient is spatially Lipschitz. For any coupled solution pair satisfying Assumption~\ref{ass:solutions},
\[
  B=0,
  \qquad
  A^2=\int_0^T |r(t)|^2\,\E \phi_\theta(X_t)^2\dd t,
  \qquad
  S=0.
\]
Moreover, $A<\infty$ and
\[
  A^2\le C_{\theta,T}\|r\|_{L^2(0,T)}^2,
\]
where $C_{\theta,T}$ depends on the baseline moment $\sup_{0\le t\le T}\E|X_t|^{2\theta}$, which is finite by Proposition~\ref{prop:baseline-moments}. Consequently, for every $R>0$ there exists $C_R>0$ such that, whenever $\|r\|_{L^2(0,T)}\le R$, $x_0=\wt x_0$, and $A<1$ (for instance, for sufficiently small $\|r\|_{L^2(0,T)}$),
\[
  \sup_{0\le t\le T}\E|X_t-\wt X_t|^{\alpha-1}
  \le C_R\|r\|_{L^2(0,T)}^{\frac{2(\alpha\wt\eta-1)}{\alpha\wt\eta-\alpha+2}}.
\]
\end{proposition}

\begin{proof}
Since $0<\theta<\alpha/2<1$, the derivative
\[
  \phi_\theta'(x)=\theta x(1+x^2)^{\theta/2-1}
\]
is bounded on $\R$. Thus $\wt a(t,\cdot)$ is Lipschitz with modulus $|r(t)|\|\phi_\theta'\|_\infty$, which belongs to $L^2(0,T)$. The stable coefficient has the same H\"older and non-Lipschitz properties as in Proposition~\ref{prop:nonlipschitz-example}. Because the drift coefficients coincide, $B=0$. The Brownian coefficient difference is $r(t)\phi_\theta(x)$, which gives the displayed expression for $A^2$; because the stable coefficients coincide, $S=0$.

For $0<\theta<\alpha/2$,
\[
  \phi_\theta(x)^2=(1+x^2)^\theta\le C_\theta(1+|x|^{2\theta}).
\]
Proposition~\ref{prop:baseline-moments} gives $\sup_{0\le t\le T}\E|X_t|^{2\theta}<\infty$, because $2\theta<\alpha$. Therefore
\[
  A^2\le C_\theta\int_0^T |r(t)|^2\bigl(1+\sup_{0\le s\le T}\E|X_s|^{2\theta}\bigr)\dd t
  \le C_{\theta,T}\|r\|_{L^2(0,T)}^2.
\]
The estimate follows from the $A$-term in Theorem~\ref{thm:noncritical}; the constant is uniform for $\|r\|_{L^2(0,T)}\le R$ because the Lipschitz modulus of $\wt a$ is bounded by a constant multiple of $R$.
\end{proof}

\begin{remark}[Moment threshold in the Brownian error]
Proposition~\ref{prop:mixed-tail-brownian-error} shows how the quadratic Brownian coefficient distance interacts with the stable tails of the baseline law. The verification of $A<\infty$ uses the moment of order $2\theta$ of the baseline process. This is exactly the moment condition used in Corollary~\ref{cor:admissible-distances}. Thus the condition $2\theta<\alpha$ appears at the level of the law-weighted Brownian error, even though the Brownian coefficient remains spatially Lipschitz.
\end{remark}

\begin{remark}[Critical endpoint]
Propositions~\ref{prop:nonlipschitz-example},~\ref{prop:mixed-brownian-error},~\ref{prop:all-errors-example}, and~\ref{prop:mixed-tail-brownian-error} are stated in the non-critical range $\wt\eta>1/\alpha$. The critical estimate of Theorem~\ref{thm:critical} applies to any coupled pair satisfying Assumption~\ref{ass:solutions} at $\wt\eta=1/\alpha$.
\end{remark}

\section{Main results}
\label{sec:main-results}

The next proposition is the estimate used in all subsequent rate statements. The localization parameters $\delta$ and $\eps$ are kept explicit because different terms are optimized against different powers of these parameters. The drift and stable direct-error terms use the first derivative of the mollifier and the localized stable-generator identity, whereas the Brownian direct-error term uses the second derivative. This is why the $A$-term carries $\delta^2/\log\delta$, rather than $\delta/\log\delta$, in the localized estimate.

\begin{proposition}[Fundamental estimate with localization parameter]\label{prop:fundamental-delta}
Assume Assumptions~\ref{ass:baseline}--\ref{ass:solutions}. Then there exists a constant
\[
C=C\bigl(T,\alpha,\wt\eta,\|\ell_b\|_{L^1(0,T)},\|\ell_a\|_{L^2(0,T)},\|\ell_\sigma\|_{L^\alpha(0,T)}\bigr)
\]
independent of $\delta$ and $\eps$ such that, for every $\delta\ge2$ and $\eps\in(0,1)$,
\begin{align}
\sup_{0\le t\le T}\E|X_t-\wt X_t|^{\alpha-1}
\le C\Bigg[&|x_0-\wt x_0|^{\alpha-1}
+\eps^{\alpha-1}
+\frac{\delta}{\log\delta}\frac{B}{\eps^{2-\alpha}} \notag\\
&+\frac{\delta^2}{\log\delta}\frac{A^2}{\eps^{3-\alpha}}
+\frac{\delta}{\log\delta}\frac{S^\alpha}{\eps}
+\frac{\delta^2}{\log\delta}\eps^{\alpha-1}
+\frac{\eps^{\alpha\wt\eta-1}}{\log\delta}
\Bigg]. \label{eq:fundamental-delta}
\end{align}
\end{proposition}

In the non-critical case $\wt\eta>1/\alpha$, the last term in \eqref{eq:fundamental-delta} is a positive power of $\eps$. Thus $\delta$ can be fixed, and only $\eps$ has to be optimized.

\begin{theorem}[Non-critical power-rate estimate]\label{thm:noncritical}
Assume Assumptions~\ref{ass:baseline}--\ref{ass:solutions}. Suppose $B,A,S<1$.
If $\wt\eta\in(1/\alpha,1]$, then there exists $C>0$ such that
\begin{align}
\sup_{0\le t\le T}\E|X_t-\wt X_t|^{\alpha-1}
\le C\Bigg[&|x_0-\wt x_0|^{\alpha-1} \notag\\
&+\max\left\{
B^{\frac{\alpha\wt\eta-1}{\alpha\wt\eta-\alpha+1}},
A^{\frac{2(\alpha\wt\eta-1)}{\alpha\wt\eta-\alpha+2}},
S^{\alpha-1/\wt\eta}
\right\}\Bigg]. \label{eq:noncritical-rate}
\end{align}
\end{theorem}

At the endpoint $\wt\eta=1/\alpha$, the H\"older contribution is independent of $\eps$. The parameter $\delta$ is then optimized together with the coefficient error, which leads to a logarithmic rate.

\begin{theorem}[Critical logarithmic estimate]\label{thm:critical}
Assume Assumptions~\ref{ass:baseline}--\ref{ass:solutions} and $B,A,S<1$. Suppose $\wt\eta=1/\alpha$ and set
\[
  \Theta:=\max\{B,A,S\}.
\]
Then there exists $C>0$ such that, for $0<\Theta<1$,
\begin{equation}\label{eq:critical-rate}
\sup_{0\le t\le T}\E|X_t-\wt X_t|^{\alpha-1}
\le C\left[|x_0-\wt x_0|^{\alpha-1}+\left(\log\frac1\Theta\right)^{-1}\right].
\end{equation}
When $\Theta=0$, the right-hand side is understood as $C|x_0-\wt x_0|^{\alpha-1}$ by the limiting convention.
\end{theorem}

\begin{proof}[Proof of \Cref{thm:noncritical}]
If $\max\{B,A,S\}=0$, take $\delta=2$ in \eqref{eq:fundamental-delta} and let $\eps\downarrow0$ to obtain the estimate. We therefore assume $\max\{B,A,S\}>0$. Take $\delta=2$ in \eqref{eq:fundamental-delta}. Set $\beta:=\alpha\wt\eta-1>0$. Since $\wt\eta\le1$, one has $\beta\le\alpha-1$, and therefore $\eps^{\alpha-1}\le\eps^\beta$ for $\eps\in(0,1)$. The relevant terms are
\[
  \frac{B}{\eps^{2-\alpha}},\qquad
  \frac{A^2}{\eps^{3-\alpha}},\qquad
  \frac{S^\alpha}{\eps},\qquad
  \eps^\beta.
\]
Balancing each of the first three terms with $\eps^\beta$ gives
\[
  \eps_B=B^{1/(\alpha\wt\eta-\alpha+1)},\qquad
  \eps_A=A^{2/(\alpha\wt\eta-\alpha+2)},\qquad
  \eps_S=S^{1/\wt\eta}.
\]
Choosing $\eps=\max\{\eps_B,\eps_A,\eps_S\}$ in \eqref{eq:fundamental-delta} yields \eqref{eq:noncritical-rate}. Indeed, $\eps\ge\eps_B$ gives
\[
  \frac{B}{\eps^{2-\alpha}}\le \frac{B}{\eps_B^{2-\alpha}}=\eps_B^\beta\le \eps^\beta,
\]
and the estimates for the $A$- and $S$-terms are identical with $\eps_A$ and $\eps_S$. The term $\eps^\beta$ then gives the three powers displayed in \eqref{eq:noncritical-rate}.
\end{proof}

\begin{proof}[Proof of \Cref{thm:critical}]
Let $\wt\eta=1/\alpha$. Set
\[
  r:=\frac{\alpha-1}{4}.
\]
We first treat the zero-error case. Suppose $\Theta=0$, so that $B=A=S=0$. For $\eps\in(0,1)$ put
\[
  \delta:=\max\{2,\eps^{-r}\}.
\]
Then $\delta\to\infty$ as $\eps\downarrow0$. Applying \eqref{eq:fundamental-delta} with $B=A=S=0$ gives
\[
\sup_{0\le t\le T}\E|X_t-\wt X_t|^{\alpha-1}
\le C\left[|x_0-\wt x_0|^{\alpha-1}+\eps^{\alpha-1}
+\frac{\delta^2}{\log\delta}\eps^{\alpha-1}
+\frac{1}{\log\delta}\right].
\]
For small $\eps$, $\delta=\eps^{-r}$, and therefore
\[
  \frac{\delta^2}{\log\delta}\eps^{\alpha-1}
  =\frac{\eps^{\alpha-1-2r}}{r\log(1/\eps)}\longrightarrow0,
\]
because $\alpha-1-2r=(\alpha-1)/2>0$. Also $1/\log\delta\to0$ and $\eps^{\alpha-1}\to0$. Letting $\eps\downarrow0$ proves the limiting convention in the case $\Theta=0$.

For $\Theta\in(0,1)$ set
\[
  \eps:=\Theta,
  \qquad
  \delta:=\max\{2,\Theta^{-r}\}.
\]
If $\Theta^{-r}<2$, then $\Theta\ge 2^{-1/r}$ and $(\log(1/\Theta))^{-1}$ is bounded below by a positive constant depending only on $r$; the estimate is therefore absorbed by increasing the final constant. It remains to consider $\Theta^{-r}\ge2$, for which $\delta=\Theta^{-r}$. Since $B,A,S\le\Theta$, the terms in \eqref{eq:fundamental-delta} are bounded as follows:
\begin{align*}
\frac{\delta B}{\eps^{2-\alpha}\log\delta}
&\le \frac{\Theta^{\alpha-1-r}}{r\log(1/\Theta)},\\
\frac{\delta^2 A^2}{\eps^{3-\alpha}\log\delta}
&\le \frac{\Theta^{\alpha-1-2r}}{r\log(1/\Theta)},\\
\frac{\delta S^\alpha}{\eps\log\delta}
&\le \frac{\Theta^{\alpha-1-r}}{r\log(1/\Theta)},\\
\frac{\delta^2\eps^{\alpha-1}}{\log\delta}
&=\frac{\Theta^{\alpha-1-2r}}{r\log(1/\Theta)},\\
\frac{\eps^{\alpha\wt\eta-1}}{\log\delta}
&=\frac{1}{r\log(1/\Theta)}.
\end{align*}
The exponents $\alpha-1-r$ and $\alpha-1-2r$ are positive for this choice of $r$. Also, for any $p>0$, $\Theta^p\le C_p(\log(1/\Theta))^{-1}$ on $(0,1)$. Thus all remaining $\eps$-dependent terms are bounded by a constant multiple of $(\log(1/\Theta))^{-1}$. This proves \eqref{eq:critical-rate}.
\end{proof}

\begin{corollary}[Synchronous pathwise uniqueness]\label{cor:pathwise-uniqueness}
Suppose the coefficients $b,a,\sigma:[0,T]\times\R\to\R$ satisfy Assumption~\ref{ass:baseline} and also the spatial regularity conditions of Assumption~\ref{ass:perturbed}. Let $X^1$ and $X^2$ be two c\`adl\`ag adapted solutions of the same equation, driven by the same pair $(W,Z)$ and starting from the same initial value. If the stochastic integrals are well defined, then $X^1$ and $X^2$ are indistinguishable.
\end{corollary}

\begin{proof}
Apply Theorem~\ref{thm:noncritical} when $\wt\eta>1/\alpha$ and Theorem~\ref{thm:critical} when $\wt\eta=1/\alpha$, taking $X^1$ as the baseline solution and $X^2$ as the perturbed solution. The initial error and all coefficient distances are zero. Hence
\[
  \sup_{0\le t\le T}\E|X_t^1-X_t^2|^{\alpha-1}=0.
\]
Thus $X_t^1=X_t^2$ almost surely for each fixed $t$. Equality on a countable dense subset of $[0,T]$, together with the c\`adl\`ag property, gives indistinguishability.
\end{proof}

\begin{remark}[Meaning of the $L^{\alpha-1}$ estimate]
The estimate has the following probabilistic consequence for every $\alpha\in(1,2)$. For any fixed threshold $\eta_0>0$, Markov's inequality gives
\[
  \sup_{0\le t\le T}\mathbb P\bigl(|X_t-\wt X_t|>\eta_0\bigr)
  \le \eta_0^{-(\alpha-1)}
  \sup_{0\le t\le T}\E|X_t-\wt X_t|^{\alpha-1}.
\]
Thus convergence in the $L^{\alpha-1}$ scale implies convergence in probability uniformly over deterministic times. The stronger estimate for the path supremum, namely a bound on
\[
  \mathbb P\left(\sup_{0\le t\le T}|X_t-\wt X_t|>\eta_0\right),
\]
is proved in Theorem~\ref{thm:probability} by a stopped quasi-martingale argument.
\end{remark}

\section{Constant-coefficient consistency checks}
\label{sec:consistency}

The exponents in Theorem~\ref{thm:noncritical} are consistent with elementary scaling in the Lipschitz case. In the Lipschitz case $\wt\eta=1$, the estimate becomes
\[
  |x_0-\wt x_0|^{\alpha-1},\qquad
  B^{\alpha-1},\qquad A^{\alpha-1},\qquad S^{\alpha-1}.
\]
The following constant-coefficient examples identify these powers with exact scalings. They also give lower bounds showing that, within these one-parameter subclasses, the Lipschitz-case exponent $\alpha-1$ cannot be improved separately in any of the four error variables.

\begin{proposition}[Consistency of the Lipschitz-case powers]
\label{prop:constant-consistency}
Let $p=\alpha-1\in(0,1)$.
\begin{enumerate}[label=\textup{(\roman*)}]
\item \textup{Initial value perturbation.} Let
\[
  X_t=x_0+W_t+Z_t,
  \qquad
  \wt X_t=\wt x_0+W_t+Z_t.
\]
Then
\[
  \sup_{0\le t\le T}\E|X_t-\wt X_t|^p=|x_0-\wt x_0|^p.
\]
Thus the initial error exponent $p=\alpha-1$ is exact.

\item \textup{Drift perturbation.} Let
\[
  X_t=bt,
  \qquad
  \wt X_t=\wt b t.
\]
Then
\[
  \sup_{0\le t\le T}\E|X_t-\wt X_t|^p=T^p|b-\wt b|^p.
\]
Since $B=T|b-\wt b|$, this is exactly $B^p$.

\item \textup{Brownian diffusion perturbation.} Let
\[
  X_t=aW_t,
  \qquad
  \wt X_t=\wt a W_t.
\]
Then
\[
  \sup_{0\le t\le T}\E|X_t-\wt X_t|^p
  =m_p T^{p/2}|a-\wt a|^p,
\]
where $m_p=\E|N(0,1)|^p$. Since $A=T^{1/2}|a-\wt a|$, this is $m_p A^p$.

\item \textup{Stable jump coefficient perturbation.} Let
\[
  X_t=\sigma Z_t,
  \qquad
  \wt X_t=\wt\sigma Z_t.
\]
Then
\[
  \sup_{0\le t\le T}\E|X_t-\wt X_t|^p
  =\E|Z_1|^p T^{p/\alpha}|\sigma-\wt\sigma|^p.
\]
Since $S=T^{1/\alpha}|\sigma-\wt\sigma|$, this is $\E|Z_1|^pS^p$.
\end{enumerate}
Consequently, when $\wt\eta=1$, the powers $|x_0-\wt x_0|^{\alpha-1}$, $B^{\alpha-1}$, $A^{\alpha-1}$ and $S^{\alpha-1}$ agree with the exact scaling of the four elementary perturbation mechanisms.
\end{proposition}

\begin{proof}
The initial-value and drift cases are deterministic after the common noise cancels. For the Brownian case, use $W_t\stackrel{d}{=}t^{1/2}N(0,1)$. For the stable case, use the stable scaling $Z_t\stackrel{d}{=}t^{1/\alpha}Z_1$ and the fact that $p=\alpha-1<\alpha$, so $\E|Z_1|^p<\infty$.
\end{proof}

The preceding proposition gives exact scalings for four one-parameter subclasses. The next proposition reformulates those scalings as a lower-bound statement: in each subclass, no uniform estimate can hold with a strictly larger power of the single nonzero error variable.

\begin{proposition}[Lower bound in the Lipschitz constant-coefficient subclass]
\label{prop:constant-sharpness}
Let $p=\alpha-1$ and fix $T>0$. In each of the four constant-coefficient families in Proposition~\ref{prop:constant-consistency}, there do not exist constants $C>0$, $q>p$ and $E_0>0$ such that
\[
  \sup_{0\le t\le T}\E|X_t-\wt X_t|^p\le C E^q
\]
for all $0<E<E_0$, where $E$ denotes the corresponding single nonzero error variable among $|x_0-\wt x_0|$, $B$, $A$, and $S$. In particular, the exponent $\alpha-1$ in the Lipschitz case $\wt\eta=1$ is optimal separately for the initial value, drift, Brownian diffusion, and stable jump coefficient errors within this constant-coefficient subclass.
\end{proposition}

\begin{proof}
It suffices to prove the Brownian case; the other three cases follow identically from Proposition~\ref{prop:constant-consistency}. In the Brownian family,
\[
  A=T^{1/2}|a-\wt a|
  \qquad\text{and}\qquad
  \sup_{0\le t\le T}\E|X_t-\wt X_t|^p=m_pA^p.
\]
If a uniform estimate with exponent $q>p$ held for every $A\in(0,E_0)$, then
\[
  m_pA^p\le CA^q.
\]
Dividing by $A^p$ and letting $A\downarrow0$ would give $m_p\le0$, contradicting $m_p>0$. The initial-value, drift, and stable cases follow by replacing $A$ with $|x_0-\wt x_0|$, $B$, and $S$, respectively.
\end{proof}

\begin{remark}[Scope of the lower bound]
Proposition~\ref{prop:constant-sharpness} establishes sharpness in the Lipschitz constant-coefficient subclass. For H\"older stable coefficients with $\wt\eta<1$, the exponents
\[
\frac{\alpha\wt\eta-1}{\alpha\wt\eta-\alpha+1},
\qquad
\frac{2(\alpha\wt\eta-1)}{\alpha\wt\eta-\alpha+2},
\qquad
\alpha-\frac1{\wt\eta}
\]
come from balancing the direct coefficient errors against the H\"older path-difference term $\eps^{\alpha\wt\eta-1}$ in the Komatsu estimate. Lower-bound estimates in that regime require coefficient-specific local constructions and are distinct from the constant-coefficient checks above.
\end{remark}

\section{Second-order Komatsu estimates}
\label{sec:auxiliary}

This section proves the analytic estimates for the mollified Riesz kernel used in Proposition~\ref{prop:fundamental-delta}. The mollifier is even, has total mass one, and is supported on the logarithmic annulus $\eps/\delta<|x|<\eps$.

\begin{lemma}[Explicit log-annular mollifier]
\label{lem:explicit-mollifier}
Let $\varphi\in C_c^\infty((0,1))$ satisfy
\[
  \varphi\ge0,\qquad \int_0^1\varphi(r)\dd r=1.
\]
For $\delta\ge2$, $\eps\in(0,1)$, put $L_\delta:=\log\delta$. For $y\ne0$, define
\[
  r_{\delta,\eps}(y):=\frac{\log(\eps/|y|)}{L_\delta}
\]
and
\begin{equation}\label{eq:explicit-psi}
  \psi_{\delta,\eps}(y)
  :=\frac{1}{2L_\delta |y|}\,\varphi\!\left(r_{\delta,\eps}(y)\right),
  \qquad y\ne0,
\end{equation}
while setting $\psi_{\delta,\eps}(0)=0$. Then
\[
  \psi_{\delta,\eps}\in C_c^\infty(\R),\qquad
  \psi_{\delta,\eps}\ge0,
  \qquad
  \int_\R\psi_{\delta,\eps}(y)\dd y=1,
\]
and
\[
  \supp\psi_{\delta,\eps}\subset\{y:\eps/\delta<|y|<\eps\}.
\]
Moreover, there exists a constant $C_\varphi>0$, independent of $\delta$ and $\eps$, such that
\begin{align}
  \|\psi_{\delta,\eps}\|_\infty
  &\le C_\varphi\frac{\delta}{\eps\log\delta}, \label{eq:psi0-bound}\\
  \|\psi_{\delta,\eps}'\|_\infty
  &\le C_\varphi\frac{\delta^2}{\eps^2\log\delta}, \label{eq:psi1-bound}\\
  \|\psi_{\delta,\eps}''\|_\infty
  &\le C_\varphi\frac{\delta^3}{\eps^3\log\delta}. \label{eq:psi2-bound}
\end{align}
Finally,
\begin{equation}\label{eq:psi-log-bound}
  0\le \psi_{\delta,\eps}(y)
  \le \frac{C_\varphi}{|y|\log\delta},
  \qquad \eps/\delta<|y|<\eps.
\end{equation}
\end{lemma}

\begin{proof}
The support assertion follows from $\supp\varphi\subset(0,1)$, since
\[
  0<r_{\delta,\eps}(y)<1
  \quad\Longleftrightarrow\quad
  \eps/\delta<|y|<\eps.
\]
The function is smooth because $\varphi$ vanishes in a neighborhood of $0$ and $1$; hence the apparent singularity at $y=0$ and the two annular endpoints do not create boundary singularities.

The normalization is as follows. By evenness and the change of variables
\[
  r=\frac{\log(\eps/y)}{L_\delta},
  \qquad y=\eps e^{-L_\delta r},
  \qquad \frac{\dd y}{y}=-L_\delta\dd r,
\]
one obtains
\[
\begin{aligned}
\int_\R\psi_{\delta,\eps}(y)\dd y
&=2\int_{\eps/\delta}^{\eps}
\frac{1}{2L_\delta y}\varphi\!\left(\frac{\log(\eps/y)}{L_\delta}\right)\dd y \\
&=\int_0^1\varphi(r)\dd r=1.
\end{aligned}
\]
The pointwise estimate \eqref{eq:psi-log-bound} follows directly from \eqref{eq:explicit-psi}. Since $|y|\ge\eps/\delta$ on the support, it implies \eqref{eq:psi0-bound}.

For $y>0$, set $r=r_{\delta,\eps}(y)$. Differentiating gives
\[
  \psi_{\delta,\eps}'(y)
  =-\frac{1}{2L_\delta y^2}\varphi(r)
   -\frac{1}{2L_\delta^2y^2}\varphi'(r).
\]
Since $\delta\ge2$, one has $L_\delta\ge\log2$; hence powers $L_\delta^{-m}$ with $m\ge1$ are bounded by a constant multiple of $L_\delta^{-1}$. The same bound holds on $y<0$ by evenness. Hence
\[
  |\psi_{\delta,\eps}'(y)|
  \le \frac{C}{L_\delta |y|^2}
  \le C\frac{\delta^2}{\eps^2\log\delta},
\]
which proves \eqref{eq:psi1-bound}. A second differentiation yields a finite linear combination of terms of the form
\[
  \frac{1}{L_\delta^m |y|^3}\varphi^{(j)}(r),
  \qquad 1\le m\le3,\quad 0\le j\le2.
\]
Therefore
\[
  |\psi_{\delta,\eps}''(y)|
  \le \frac{C}{L_\delta |y|^3}
  \le C\frac{\delta^3}{\eps^3\log\delta},
\]
which proves \eqref{eq:psi2-bound}.
\end{proof}

\begin{definition}[Mollified Riesz kernel]\label{def:mollifier-delta}
Let $\psi_{\delta,\eps}$ be the log-annular mollifier from Lemma~\ref{lem:explicit-mollifier}. Define
\[
  u_{\delta,\eps}(x):=(|\cdot|^{\alpha-1}*\psi_{\delta,\eps})(x),
  \qquad x\in\R.
\]
\end{definition}

\begin{lemma}[Approximation]\label{lem:approx}
For all $x\in\R$, $\delta\ge2$, and $\eps\in(0,1)$,
\begin{align}
  |x|^{\alpha-1}&\le u_{\delta,\eps}(x)+\eps^{\alpha-1}, \label{eq:approx-lower}\\
  u_{\delta,\eps}(x)&\le |x|^{\alpha-1}+\eps^{\alpha-1}. \label{eq:approx-upper}
\end{align}
\end{lemma}

\begin{proof}
Since $0<\alpha-1<1$, the map $r\mapsto r^{\alpha-1}$ is subadditive on $[0,\infty)$. On $\supp\psi_{\delta,\eps}$ one has $|y|\le\eps$. Hence
\[
  |x|^{\alpha-1}\le |x-y|^{\alpha-1}+|y|^{\alpha-1}\le |x-y|^{\alpha-1}+\eps^{\alpha-1},
\]
and integration gives \eqref{eq:approx-lower}. The proof of \eqref{eq:approx-upper} is identical.
\end{proof}

\begin{proposition}[First- and second-order Komatsu estimates]\label{prop:komatsu}
There exists $C=C(\alpha,\varphi)>0$ such that, for all $x\in\R$, $\delta\ge2$, and $\eps\in(0,1)$,
\begin{align}
|u_{\delta,\eps}'(x)|
&\le C\left(|x|^{\alpha-2}\1_{\{|x|>2\eps\}}
+\frac{\delta}{\eps^{2-\alpha}\log\delta}\1_{\{|x|\le2\eps\}}\right), \label{eq:delta-first}\\
|u_{\delta,\eps}'(x)|\,|x|
&\le C\left(|x|^{\alpha-1}+\frac{\delta}{\log\delta}\eps^{\alpha-1}\right), \label{eq:delta-first-product}\\
|u_{\delta,\eps}''(x)|
&\le C\left(|x|^{\alpha-3}\1_{\{|x|>2\eps\}}
+\frac{\delta^2}{\eps^{3-\alpha}\log\delta}\1_{\{|x|\le2\eps\}}\right), \label{eq:delta-second}\\
|u_{\delta,\eps}''(x)|x^2
&\le C\left(|x|^{\alpha-1}+\frac{\delta^2}{\log\delta}\eps^{\alpha-1}\right). \label{eq:delta-second-product}
\end{align}
In particular,
\begin{equation}\label{eq:global-derivative-bounds}
  \|u_{\delta,\eps}'\|_\infty\le C\frac{\delta}{\eps^{2-\alpha}\log\delta},
  \qquad
  \|u_{\delta,\eps}''\|_\infty\le C\frac{\delta^2}{\eps^{3-\alpha}\log\delta}.
\end{equation}
\end{proposition}

\begin{proof}
Let $\RieszK(x)=|x|^{\alpha-1}$. If $|x|>2\eps$ and $y\in\supp\psi_{\delta,\eps}$, then
\[
  \frac{|x|}{2}\le |x-y|\le \frac{3|x|}{2}.
\]
Differentiating under the integral gives
\[
  u_{\delta,\eps}'(x)=(\alpha-1)\int_\R \sgn(x-y)|x-y|^{\alpha-2}\psi_{\delta,\eps}(y)\dd y,
\]
so $|u_{\delta,\eps}'(x)|\le C|x|^{\alpha-2}$. Similarly,
\[
  u_{\delta,\eps}''(x)=(\alpha-1)(\alpha-2)\int_\R |x-y|^{\alpha-3}\psi_{\delta,\eps}(y)\dd y,
\]
so $|u_{\delta,\eps}''(x)|\le C|x|^{\alpha-3}$.

Next suppose $|x|\le2\eps$. For the first derivative, using the integral formula above and the pointwise bound \eqref{eq:psi-log-bound},
\begin{align*}
|u_{\delta,\eps}'(x)|
&\le C\int_{\eps/\delta<|y|<\eps}|x-y|^{\alpha-2}\psi_{\delta,\eps}(y)\dd y\\
&\le C\frac{\delta}{\eps\log\delta}
\int_{\eps/\delta<|y|<\eps}|x-y|^{\alpha-2}\dd y.
\end{align*}
The last integral is bounded by $C\eps^{\alpha-1}$ because $\alpha-2>-1$ and, when $|x|\le2\eps$ and $|y|<\eps$, the variable $x-y$ ranges inside $[-3\eps,3\eps]$. Hence
\[
|u_{\delta,\eps}'(x)|\le C\frac{\delta}{\eps^{2-\alpha}\log\delta}.
\]
For the second derivative near the origin, use the convolution identity
\[
  u_{\delta,\eps}''=\RieszK' * \psi_{\delta,\eps}',
  \qquad \RieszK'(x)=(\alpha-1)\sgn(x)|x|^{\alpha-2}.
\]
Then \eqref{eq:psi1-bound} gives
\begin{align*}
|u_{\delta,\eps}''(x)|
&\le C\frac{\delta^2}{\eps^2\log\delta}
\int_{\eps/\delta<|y|<\eps}|x-y|^{\alpha-2}\dd y\\
&\le C\frac{\delta^2}{\eps^2\log\delta}\eps^{\alpha-1}
= C\frac{\delta^2}{\eps^{3-\alpha}\log\delta}.
\end{align*}
The product estimates follow by multiplying the preceding bounds by $|x|$ or $x^2$ and separating the two regions $|x|>2\eps$ and $|x|\le2\eps$. For instance, if $|x|\le2\eps$, then
\[
  |u_{\delta,\eps}''(x)|x^2
  \le C\frac{\delta^2}{\eps^{3-\alpha}\log\delta}\eps^2
  =C\frac{\delta^2}{\log\delta}\eps^{\alpha-1},
\]
whereas the far-field estimate gives $|u_{\delta,\eps}''(x)|x^2\le C|x|^{\alpha-1}$ when $|x|>2\eps$.
\end{proof}

\begin{lemma}[Stable generator identity]\label{lem:stable-generator}
Let $\mathcal S(\R)$ denote the Schwartz space and let $\mathcal S'(\R)$ denote the space of tempered distributions. Let $L^\alpha$ be the generator of the one-dimensional symmetric $\alpha$-stable process, with sign convention
\[
L^\alpha f(x)=\int_{\R\setminus\{0\}}\left(f(x+z)-f(x)-\1_{\{|z|\le1\}}zf'(x)\right)c_\alpha |z|^{-1-\alpha}\dd z.
\]
Then
\[
  L^\alpha u_{\delta,\eps}(x)=C_\alpha\psi_{\delta,\eps}(x)
\]
in the sense of distributions and pointwise for the mollified function, for some constant $C_\alpha>0$. Consequently, for a scale parameter $\gamma\in\R$,
\[
  L^\alpha_\gamma u_{\delta,\eps}(x)=C_\alpha |\gamma|^\alpha\psi_{\delta,\eps}(x),
\]
where
\[
L^\alpha_\gamma f(x):=\int_{\R\setminus\{0\}}\left(f(x+\gamma z)-f(x)-\1_{\{|z|\le1\}}\gamma z f'(x)\right)\nu(\md z).
\]
\end{lemma}

\begin{proof}
Let $\RieszK(x)=|x|^{\alpha-1}$. The details are included because this identity is the point at which the singular Riesz kernel enters the argument.

\emph{Step 1: the Riesz identity in distributions.} With the Fourier convention under which
\[
  \widehat{L^\alpha f}(\xi)=-\widehat c_\alpha |\xi|^\alpha\widehat f(\xi),
  \qquad \widehat c_\alpha>0,
\]
the homogeneous-distribution formula gives
\[
  \widehat{\RieszK}(\xi)=-A_\alpha |\xi|^{-\alpha},
  \qquad A_\alpha>0.
\]
One obtains this formula by analytic continuation of the Fourier transform of $|x|^\lambda$, $0<\lambda<1$, and then setting $\lambda=\alpha-1$; see the classical references on homogeneous distributions and Riesz kernels \cite{GelfandShilov1964,Landkof1972}, and also the survey \cite{Kwasnicki2017}. Consequently,
\[
  \widehat{L^\alpha\RieszK}(\xi)=\widehat c_\alpha A_\alpha,
\]
and hence
\begin{equation}\label{eq:riesz-identity}
  L^\alpha\RieszK=C_\alpha\delta_0,
  \qquad C_\alpha>0,
\end{equation}
in $\mathcal S'(\R)$. The positivity of $C_\alpha$ is consistent with the limiting identity $\Delta |x|=2\delta_0$ when $\alpha=2$. Only the positivity and finiteness of this constant are used below.

\emph{Step 2: transfer to the mollified kernel.} Let $\chi\in\mathcal S(\R)$. The function $L^\alpha\chi$ decays at least of order $|x|^{-1-\alpha}$ at infinity, while $\RieszK(x)$ grows like $|x|^{\alpha-1}$; hence $\RieszK(x)L^\alpha\chi(x)$ is integrable at infinity. The stated decay follows from the Fourier multiplier representation of $L^\alpha$ on Schwartz functions, or directly by splitting the L\'evy integral into near and far regions. More explicitly, $\RieszK\in L^1_{\rm loc}(\R)$ has polynomial growth of order $\alpha-1$, whereas $\psi_{\delta,\eps}*L^\alpha\chi$ and $L^\alpha(\psi_{\delta,\eps}*\chi)$ decay at least as $O(|x|^{-1-\alpha})$; their products with $\RieszK$ therefore decay like $O(|x|^{-2})$ at infinity. Thus all pairings in the following display are ordinary absolutely convergent Lebesgue integrals. We compute the distribution $L^\alpha(\RieszK*\psi_{\delta,\eps})$ by duality against $\chi$. The commutation relation $\psi_{\delta,\eps}*L^\alpha\chi=L^\alpha(\psi_{\delta,\eps}*\chi)$ follows on $\mathcal S(\R)$ from the Fourier multiplier representation of $L^\alpha$, and the symmetry of the L\'evy measure gives the self-adjointness used below.
\begin{align*}
\langle L^\alpha(\RieszK*\psi_{\delta,\eps}),\chi\rangle
&=\langle \RieszK*\psi_{\delta,\eps},L^\alpha\chi\rangle\\
&=\langle \RieszK,\psi_{\delta,\eps}*L^\alpha\chi\rangle\\
&=\langle \RieszK,L^\alpha(\psi_{\delta,\eps}*\chi)\rangle\\
&=C_\alpha(\psi_{\delta,\eps}*\chi)(0).
\end{align*}
Here the third line uses that $\psi_{\delta,\eps}*\chi\in\mathcal S(\R)$ and that $L^\alpha$ is a Fourier multiplier. Since $\psi_{\delta,\eps}$ is even,
\[
  (\psi_{\delta,\eps}*\chi)(0)=\int_{\R}\psi_{\delta,\eps}(y)\chi(y)\dd y.
\]
Thus
\[
  \langle L^\alpha u_{\delta,\eps},\chi\rangle
  =\langle C_\alpha\psi_{\delta,\eps},\chi\rangle.
\]
This proves the distributional identity. The argument avoids a pointwise interchange of $L^\alpha$ and convolution at the singular kernel $\RieszK$, where the singularity of the unmollified kernel has to be treated distributionally.

\emph{Step 3: pointwise interpretation.} \Cref{prop:komatsu} gives $u_{\delta,\eps}\in C^2$ with bounded first and second derivatives for fixed $(\delta,\eps)$. Therefore the defining integral for $L^\alpha u_{\delta,\eps}(x)$ is absolutely convergent: near the origin the second-order Taylor remainder is controlled by $\|u_{\delta,\eps}''\|_\infty |z|^2$, and at infinity the growth $u_{\delta,\eps}(x)=O(1+|x|^{\alpha-1})$ is integrable against $|z|^{-1-\alpha}\dd z$. The same estimates give continuity in $x$ by dominated convergence on compact sets. Since both sides are continuous representatives of the same distribution, the equality holds pointwise.

For the scaled identity, change variables $w=\gamma z$. If $\gamma=0$, both sides vanish. If $\gamma\ne0$, the L\'evy measure is $\alpha$-homogeneous and contributes the factor $|\gamma|^\alpha$. The cutoff changes from $\1_{\{|z|\le1\}}$ to $\1_{\{|w|\le |\gamma|\}}$. The difference between this cutoff and $\1_{\{|w|\le1\}}$ contributes an odd integrand proportional to $w f'(x)$ over an annulus; since $\alpha>1$, this annular integral is absolutely convergent and vanishes by symmetry. Therefore $L^\alpha_\gamma f=|\gamma|^\alpha L^\alpha f$ for $f=u_{\delta,\eps}$, and the result follows.
\end{proof}

\section{Localization and a priori estimates}
\label{sec:localization}

For $n\ge1$, define
\[
  \tau_n:=\inf\{t\in[0,T]: |X_t|\vee |\wt X_t|\ge n\}\wedge T.
\]
Since $X$ and $\wt X$ are c\`adl\`ag on $[0,T]$, their paths are almost surely bounded on $[0,T]$, and therefore $\tau_n=T$ for all sufficiently large $n$, almost surely.

For later reference, decompose the three coefficient differences as
\begin{align}
b(s,X_s)-\wt b(s,\wt X_s)&=\Delta_s^b+\Gamma_s^b,
\label{eq:drift-coefficient-difference}\\
\Delta_s^b&:=b(s,X_s)-\wt b(s,X_s),\qquad
\Gamma_s^b:=\wt b(s,X_s)-\wt b(s,\wt X_s), \notag\\[2mm]
a(s,X_s)-\wt a(s,\wt X_s)&=\Delta_s^a+\Gamma_s^a,
\label{eq:diffusion-coefficient-difference}\\
\Delta_s^a&:=a(s,X_s)-\wt a(s,X_s),\qquad
\Gamma_s^a:=\wt a(s,X_s)-\wt a(s,\wt X_s), \notag\\[2mm]
\sigma(s,X_{s-})-\wt\sigma(s,\wt X_{s-})&=\Delta_s^\sigma+\Gamma_s^\sigma,
\label{eq:jump-coefficient-difference}\\
\Delta_s^\sigma&:=\sigma(s,X_{s-})-\wt\sigma(s,X_{s-}), \notag\\
\Gamma_s^\sigma&:=\wt\sigma(s,X_{s-})-\wt\sigma(s,\wt X_{s-}). \notag
\end{align}
In finite-variation integrals, $X_s$ and $X_{s-}$ may be interchanged because they differ only at the countable set of jump times.

The first localization tool is a maximal inequality for stable stochastic integrals. It provides the a priori finite stopped moment needed before Gronwall's inequality is applied.

\begin{lemma}[Maximal moment estimate for stable stochastic integrals]\label{lem:stable-max}
Let $\gamma=(\gamma_s)_{0\le s\le T}$ be predictable and satisfy
\[
  \E\int_0^T |\gamma_s|^\alpha\dd s<\infty.
\]
Define
\[
  M_t:=\int_0^t\int_{\R\setminus\{0\}}\gamma_s z\,\wt N(\md s,\md z).
\]
Then, for every $p\in(0,\alpha)$, there is a constant $C_{p,\alpha}$ such that
\[
  \E\sup_{0\le t\le T}|M_t|^p
  \le C_{p,\alpha}\left(\E\int_0^T |\gamma_s|^\alpha\dd s\right)^{p/\alpha}.
\]
\end{lemma}

\begin{proof}
First suppose that $\gamma$ is bounded and predictable, and that $\int_0^T|\gamma_s|^\alpha\dd s$ is bounded. The estimate proved below has a constant independent of these bounds. For a general predictable $\gamma$, set $\gamma^{m}_s:=\gamma_s\mathbf 1_{\{|\gamma_s|\le m\}}\mathbf 1_{\{\int_0^s|\gamma_r|^\alpha\dd r\le m\}}$. The bounded case applies to $\gamma^m$. The corresponding stochastic integrals converge in probability uniformly on $[0,T]$ because the same tail estimate applied to $\gamma^m-\gamma^k$ makes them Cauchy. Since $\int_0^T|\gamma_s^m-\gamma_s|^\alpha\dd s\to0$ in $L^1$ by dominated convergence, Fatou's lemma for the nonnegative variables $\sup_{0\le t\le T}|M_t^m|^p$ yields the general case. It remains to prove the estimate in the bounded localized case.

Put $K_\gamma:=\E\int_0^T|\gamma_s|^\alpha\dd s$. If $K_\gamma=0$, then the stochastic integral vanishes identically. For $\lambda>0$ decompose the jumps according to $|\gamma_s z|\le\lambda$ and $|\gamma_s z|>\lambda$:
\[
  M=M^{1,\lambda}+M^{2,\lambda}.
\]
For the small-jump part, Doob's inequality and the isometry give
\begin{align*}
\Pbb\left(\sup_{0\le t\le T}|M_t^{1,\lambda}|>\lambda\right)
&\le \lambda^{-2}\E\int_0^T\int_{|\gamma_s z|\le\lambda}|\gamma_s z|^2\nu(\md z)\dd s\\
&= C \lambda^{-2}\E\int_0^T |\gamma_s|^2\left( \frac{\lambda}{|\gamma_s|}\right)^{2-\alpha}\dd s
\le C K_\gamma\lambda^{-\alpha},
\end{align*}
with the convention that the integrand is zero on $\{\gamma_s=0\}$.

For the large-jump part, the compensator term is zero. Indeed, for fixed $s$ with $\gamma_s\ne0$,
\[
  \gamma_s\int_{|\gamma_s z|>\lambda}z\nu(\md z)=0,
\]
because the domain is symmetric and the integrand is odd; the integral is absolutely convergent since it is over $|z|>\lambda/|\gamma_s|$ and $\alpha>1$. Hence $M^{2,\lambda}$ is a non-compensated jump sum, and it is nonzero only if at least one jump satisfying $|\gamma_s z|>\lambda$ occurs. Therefore, by Markov's inequality applied to the corresponding jump count,
\[
\Pbb\left(M^{2,\lambda}\not\equiv0\text{ on }[0,T]\right)
\le \E\int_0^T\int_{|\gamma_s z|>\lambda}\nu(\md z)\dd s
\le C K_\gamma\lambda^{-\alpha}.
\]
Combining the two estimates and replacing $2\lambda$ by $\lambda$ gives
\[
  \Pbb\left(\sup_{0\le t\le T}|M_t|>\lambda\right)
  \le C K_\gamma\lambda^{-\alpha}.
\]
By the layer-cake representation, for any $\lambda_0>0$,
\begin{align*}
\E\sup_{0\le t\le T}|M_t|^p
&=\int_0^\infty p\lambda^{p-1}\Pbb\left(\sup_{0\le t\le T}|M_t|>\lambda\right)\dd\lambda\\
&\le C\lambda_0^p+C_{p,\alpha}K_\gamma\lambda_0^{p-\alpha}.
\end{align*}
Choosing $\lambda_0=K_\gamma^{1/\alpha}$ proves the estimate. The localization step described at the beginning completes the proof for general predictable $\gamma$ satisfying the stated integrability condition; see also the compensated-Poisson martingale criterion in \cite[Ch.~II]{IkedaWatanabe1989}.
\end{proof}

\begin{lemma}[A priori finite stopped moment]\label{lem:apriori-stopped}
For each $n\ge1$,
\[
  \sup_{0\le t\le T}\E|Y_{t\wedge\tau_n}|^{\alpha-1}<\infty.
\]
\end{lemma}

\begin{proof}
Write $p:=\alpha-1\in(0,1)$. Since $r\mapsto r^p$ is subadditive,
\[
  \E|Y_{t\wedge\tau_n}|^p
  \le |Y_0|^p+\E|D_t^{(n)}|^p+\E|M_t^{W,n}|^p+\E|M_t^{Z,n}|^p,
\]
where
\begin{align*}
D_t^{(n)}&=\int_0^{t\wedge\tau_n}(b(s,X_s)-\wt b(s,\wt X_s))\dd s,\\
M_t^{W,n}&=\int_0^{t\wedge\tau_n}(a(s,X_s)-\wt a(s,\wt X_s))\dd W_s,\\
M_t^{Z,n}&=\int_0^{t\wedge\tau_n}\int_{\R\setminus\{0\}}(\sigma(s,X_{s-})-\wt\sigma(s,\wt X_{s-}))z\,\wt N(\md s,\md z).
\end{align*}
For Lebesgue-a.e. $s$ with $s\le\tau_n$, and for predictable jump integrands evaluated at left limits, the stopped variables are bounded by $n$ before the exit jump. Hence the relevant path difference is bounded by $2n$. Using \eqref{eq:drift-coefficient-difference}--\eqref{eq:jump-coefficient-difference} and Assumption~\ref{ass:perturbed}, we get
\begin{align*}
|\Gamma_s^b|&\le 2n \ell_b(s),\\
|\Gamma_s^a|&\le 2n \ell_a(s),\\
|\Gamma_s^\sigma|&\le 2n \ell_\sigma(s),
\end{align*}
where we used $\wt\eta\le1$, so $(2n)^{\wt\eta}\vee 2n=2n$ for $n\ge1$.

For the drift term,
\[
  \E|D_t^{(n)}|
  \le B+2n\|\ell_b\|_{L^1(0,T)}<\infty,
\]
and Jensen's inequality for the concave function $r^p$ yields $\E|D_t^{(n)}|^p<\infty$ uniformly in $t$. For the Brownian term, It\^o's isometry gives
\[
  \E|M_t^{W,n}|^2
  \le C\left(A^2+n^2\|\ell_a\|_{L^2(0,T)}^2\right)<\infty,
\]
and again Jensen yields a finite $p$-moment. For the stable term,
\[
  \E\int_0^{T\wedge\tau_n}|\sigma(s,X_s)-\wt\sigma(s,\wt X_s)|^\alpha\dd s
  \le C\left(S^\alpha+(2n)^\alpha\|\ell_\sigma\|_{L^\alpha(0,T)}^\alpha\right)<\infty.
\]
Applying Lemma~\ref{lem:stable-max} with exponent $p=\alpha-1$ gives
\[
  \E\sup_{0\le t\le T}|M_t^{Z,n}|^p<\infty.
\]
Combining these estimates proves the claim.
\end{proof}

\begin{lemma}[Localized martingale property]\label{lem:localized-martingales}
For each fixed $\delta\ge2$, $\eps\in(0,1)$ and $n\ge1$, define
\[
M_t^{W,n}:=\int_0^{t\wedge\tau_n}u'_{\delta,\eps}(Y_s)
\bigl(a(s,X_s)-\wt a(s,\wt X_s)\bigr)\dd W_s
\]
and
\[
M_t^{Z,n}:=\int_0^{t\wedge\tau_n}\int_{\R\setminus\{0\}}
\{u_{\delta,\eps}(Y_{s-}+\gamma_s z)-u_{\delta,\eps}(Y_{s-})\}\,\wt N(\md s,\md z),
\]
where $\gamma_s:=\sigma(s,X_{s-})-\wt\sigma(s,\wt X_{s-})$. Then $M^{W,n}$ and $M^{Z,n}$ are true martingales.
\end{lemma}

\begin{proof}
For the Brownian term, \Cref{prop:komatsu} gives $\|u_{\delta,\eps}'\|_\infty<\infty$ for fixed $(\delta,\eps)$. On $[0,\tau_n]$,
\[
|a(s,X_s)-\wt a(s,\wt X_s)|^2
\le C\left(|\Delta_s^a|^2+n^2\ell_a(s)^2\right),
\]
and the right-hand side is integrable by the definition of $A$ and $\ell_a\in L^2(0,T)$. Hence the stopped Brownian stochastic integral is square-integrable.

For the jump term, set
\[
  F_s(z):=u_{\delta,\eps}(Y_{s-}+\gamma_s z)-u_{\delta,\eps}(Y_{s-}).
\]
On $[0,\tau_n]$,
\[
  |\gamma_s|^\alpha
  \le C\left(|\Delta_s^\sigma|^\alpha+(2n)^\alpha \ell_\sigma(s)^\alpha\right),
\]
which is integrable in expectation. Since $u_{\delta,\eps}$ is Lipschitz for fixed $(\delta,\eps)$,
\[
|F_s(z)|\le \|u_{\delta,\eps}'\|_\infty |\gamma_s||z|.
\]
The scaling estimate
\[
\int_{\R\setminus\{0\}}
\left((\lambda |z|)^2\wedge \lambda |z|\right)\nu(\md z)
\le C\lambda^\alpha,
\qquad \lambda\ge0,
\]
therefore implies
\[
\E\int_0^{T\wedge\tau_n}\int_{\R\setminus\{0\}}
\left(|F_s(z)|^2\wedge |F_s(z)|\right)\nu(\md z)\dd s<\infty.
\]
The compensated-Poisson integrability criterion, see \cite[Ch.~II]{IkedaWatanabe1989}, gives the true martingale property.
\end{proof}

\section{Proof of the fundamental estimate}
\label{sec:proof-fundamental}

\begin{proof}[Proof of Proposition~\ref{prop:fundamental-delta}]
Applying It\^o's formula for jump processes (for instance \cite[Theorem~4.4.7]{Applebaum2009}) to $u_{\delta,\eps}(Y_{t\wedge\tau_n})$ gives
\begin{equation}\label{eq:ito}
  u_{\delta,\eps}(Y_{t\wedge\tau_n})=u_{\delta,\eps}(Y_0)+M_{t}^{W,n}+M_{t}^{Z,n}+I_t^{b,n}+I_t^{a,n}+I_t^{\sigma,n},
\end{equation}
where
\[
I_t^{b,n}:=\int_0^{t\wedge\tau_n} u_{\delta,\eps}'(Y_s)
\bigl(b(s,X_s)-\wt b(s,\wt X_s)\bigr)\dd s,
\]
\[
I_t^{a,n}:=\frac12\int_0^{t\wedge\tau_n} u_{\delta,\eps}''(Y_s)
\bigl(a(s,X_s)-\wt a(s,\wt X_s)\bigr)^2\dd s,
\]
and, by Lemma~\ref{lem:stable-generator},
\[
I_t^{\sigma,n}:=C_\alpha\int_0^{t\wedge\tau_n} |\sigma(s,X_s)-\wt\sigma(s,\wt X_s)|^\alpha
\psi_{\delta,\eps}(Y_s)\dd s.
\]
The finite-variation jump term is the compensator term associated with the cutoff $\1_{\{|z|\le1\}}$ in the stable generator. If another deterministic cutoff is used after the scaling change of variables, the difference is an odd annular integral and is zero by the symmetry of $\nu$, as proved in Lemma~\ref{lem:stable-generator}. In all $\dd s$ integrals below, $Y_s$ and $Y_{s-}$ can be interchanged, since the set of jump times is at most countable. Lemma~\ref{lem:localized-martingales} gives that the two stopped stochastic-integral terms are true martingales.

In the following estimates, the stopping is kept explicitly in the first display of each term. In subsequent lines we use the shorthand that every finite-variation integral is multiplied by $\1_{\{s\le\tau_n\}}$. On this event, $Y_s=Y_{s\wedge\tau_n}$ for $\dd s$-a.e. $s$, and the direct coefficient-error terms are still controlled by $B,A,S$.

\subsection{Drift term}
\label{subsec:drift-term}

By \eqref{eq:drift-coefficient-difference} and \eqref{eq:b-lip},
\[
  |\Gamma_s^b|\le \ell_b(s)|Y_s|.
\]
The direct drift error is controlled by the global bound on $u'_{\delta,\eps}$. By \eqref{eq:global-derivative-bounds},
\begin{align*}
\E\int_0^t \1_{\{s\le\tau_n\}} |u_{\delta,\eps}'(Y_s)||\Delta_s^b|\dd s
&\le C\frac{\delta}{\eps^{2-\alpha}\log\delta}
\int_0^T\E |b(s,X_s)-\wt b(s,X_s)|\dd s\\
&=C\frac{\delta}{\log\delta}\frac{B}{\eps^{2-\alpha}}.
\end{align*}
The remaining drift part is recursive in the distance between the two paths. For this term, \eqref{eq:delta-first-product} gives
\begin{align*}
\E\int_0^t \1_{\{s\le\tau_n\}} |u_{\delta,\eps}'(Y_s)||\Gamma_s^b|\dd s
&\le \int_0^t \ell_b(s)\E\left[|u_{\delta,\eps}'(Y_{s\wedge\tau_n})||Y_{s\wedge\tau_n}|\right]\dd s\\
&\le C\int_0^t \ell_b(s)\E |Y_{s\wedge\tau_n}|^{\alpha-1}\dd s
+C\frac{\delta}{\log\delta}\eps^{\alpha-1}\|\ell_b\|_{L^1(0,T)}.
\end{align*}
Thus, for stopped processes,
\begin{equation}\label{eq:drift-estimate}
\E |I_{t}^{b,n}|
\le C\frac{\delta}{\log\delta}\frac{B}{\eps^{2-\alpha}}
+C\int_0^t \ell_b(s)\E|Y_{s\wedge\tau_n}|^{\alpha-1}\dd s
+C\frac{\delta}{\log\delta}\eps^{\alpha-1}.
\end{equation}

\subsection{Brownian correction term}
\label{subsec:brownian-term}

By \eqref{eq:diffusion-coefficient-difference} and \eqref{eq:a-lip},
\[
  |\Gamma_s^a|\le \ell_a(s)|Y_s|.
\]
Since $|x+y|^2\le2|x|^2+2|y|^2$,
\[
|I_t^{a,n}|\le C\int_0^{t\wedge\tau_n} |u_{\delta,\eps}''(Y_s)||\Delta_s^a|^2\dd s
+C\int_0^{t\wedge\tau_n} |u_{\delta,\eps}''(Y_s)||\Gamma_s^a|^2\dd s.
\]
The direct Brownian coefficient error is controlled by the global bound on $u''_{\delta,\eps}$. By \eqref{eq:global-derivative-bounds},
\begin{align*}
\E\int_0^{T\wedge\tau_n} |u_{\delta,\eps}''(Y_s)||\Delta_s^a|^2\dd s
&\le C\frac{\delta^2}{\eps^{3-\alpha}\log\delta}
\int_0^T\E |a(s,X_s)-\wt a(s,X_s)|^2\dd s\\
&=C\frac{\delta^2}{\log\delta}\frac{A^2}{\eps^{3-\alpha}}.
\end{align*}
The path-difference part of the Brownian correction is the reason for using the product estimate on $u''_{\delta,\eps}$. By \eqref{eq:delta-second-product},
\begin{align*}
\E\int_0^{t\wedge\tau_n} |u_{\delta,\eps}''(Y_s)||\Gamma_s^a|^2\dd s
&\le \int_0^t \ell_a(s)^2\E\left[|u_{\delta,\eps}''(Y_{s\wedge\tau_n})|Y_{s\wedge\tau_n}^2\right]\dd s\\
&\le C\int_0^t \ell_a(s)^2\E |Y_{s\wedge\tau_n}|^{\alpha-1}\dd s
+C\frac{\delta^2}{\log\delta}\eps^{\alpha-1}\int_0^t \ell_a(s)^2\dd s.
\end{align*}
Consequently,
\begin{equation}\label{eq:brownian-estimate}
\E |I_t^{a,n}|
\le C\frac{\delta^2}{\log\delta}\frac{A^2}{\eps^{3-\alpha}}
+C\int_0^t \ell_a(s)^2\E|Y_{s\wedge\tau_n}|^{\alpha-1}\dd s
+C\frac{\delta^2}{\log\delta}\eps^{\alpha-1}.
\end{equation}

\subsection{Stable jump term}
\label{subsec:stable-term}

Recall \eqref{eq:jump-coefficient-difference}. On $\{s\le\tau_n\}$ and on the support of $\psi_{\delta,\eps}(Y_s)$, one has $Y_s=Y_{s\wedge\tau_n}$, $|Y_{s\wedge\tau_n}|<\eps<1$, and therefore
\[
  |\Gamma_s^\sigma|\le \ell_\sigma(s)|Y_s|^{\wt\eta}.
\]
The term $I_t^{\sigma,n}$ is nonnegative because $C_\alpha>0$ and $\psi_{\delta,\eps}\ge0$. The stable compensator is localized by the support of $\psi_{\delta,\eps}$. Using $|x+y|^\alpha\le2^{\alpha-1}(|x|^\alpha+|y|^\alpha)$, \eqref{eq:psi0-bound}, and the definition of $S$,
\begin{align*}
\E I_t^{\sigma,n}
&\le C\E\int_0^{t\wedge\tau_n}|\Delta_s^\sigma|^\alpha\psi_{\delta,\eps}(Y_s)\dd s
+C\E\int_0^{t\wedge\tau_n}|\Gamma_s^\sigma|^\alpha\psi_{\delta,\eps}(Y_s)\dd s\\
&\le C\frac{\delta}{\eps\log\delta}S^\alpha
+C\int_0^T \ell_\sigma(s)^\alpha\E\left[|Y_{s\wedge\tau_n}|^{\alpha\wt\eta}\psi_{\delta,\eps}(Y_{s\wedge\tau_n})\right]\dd s.
\end{align*}
On the support of $\psi_{\delta,\eps}(Y_{s\wedge\tau_n})$, one has
$\eps/\delta<|Y_{s\wedge\tau_n}|<\eps$ and
\[
|Y_{s\wedge\tau_n}|^{\alpha\wt\eta}\psi_{\delta,\eps}(Y_{s\wedge\tau_n})
\le C\frac{|Y_{s\wedge\tau_n}|^{\alpha\wt\eta-1}}{\log\delta}
\le C\frac{\eps^{\alpha\wt\eta-1}}{\log\delta},
\]
where $\alpha\wt\eta-1\ge0$. Thus
\begin{equation}\label{eq:stable-estimate}
\E I_t^{\sigma,n}
\le C\frac{\delta}{\log\delta}\frac{S^\alpha}{\eps}
+C\frac{\eps^{\alpha\wt\eta-1}}{\log\delta}\|\ell_\sigma\|_{L^\alpha(0,T)}^\alpha.
\end{equation}

\subsection{Conclusion of the proof}
\label{subsec:proof-conclusion}

By Lemma~\ref{lem:approx},
\[
  |Y_{t\wedge\tau_n}|^{\alpha-1}\le u_{\delta,\eps}(Y_{t\wedge\tau_n})+\eps^{\alpha-1},
  \qquad
  u_{\delta,\eps}(Y_0)\le |Y_0|^{\alpha-1}+\eps^{\alpha-1}.
\]
Let
\[
  g_n(t):=\E|Y_{t\wedge\tau_n}|^{\alpha-1}.
\]
By Lemma~\ref{lem:apriori-stopped}, $g_n$ is finite and bounded on $[0,T]$ for each fixed $n$. Taking expectations in \eqref{eq:ito}, using Lemma~\ref{lem:localized-martingales} to remove the martingale terms, and applying \eqref{eq:drift-estimate}, \eqref{eq:brownian-estimate}, and \eqref{eq:stable-estimate}, one obtains
\begin{align*}
g_n(t)
\le C\Bigg[&|Y_0|^{\alpha-1}
+\eps^{\alpha-1}
+\frac{\delta}{\log\delta}\frac{B}{\eps^{2-\alpha}}
+\frac{\delta^2}{\log\delta}\frac{A^2}{\eps^{3-\alpha}}\\
&+\frac{\delta}{\log\delta}\frac{S^\alpha}{\eps}
+\frac{\delta^2}{\log\delta}\eps^{\alpha-1}
+\frac{\eps^{\alpha\wt\eta-1}}{\log\delta}
\Bigg]
+C\int_0^t\left(\ell_b(s)+\ell_a(s)^2\right)g_n(s)\dd s.
\end{align*}
Since $\ell_b+\ell_a^2\in L^1(0,T)$ and $g_n$ is finite, Gronwall's inequality gives the same bound uniformly in $n$. Because $\tau_n=T$ eventually almost surely, $Y_{t\wedge\tau_n}\to Y_t$ almost surely for each $t$. Fatou's lemma yields the same estimate for $\E|Y_t|^{\alpha-1}$. Taking the supremum over $t\in[0,T]$ proves Proposition~\ref{prop:fundamental-delta}.
\end{proof}

\section{Convergence in probability}
\label{sec:probability}

The expectation estimate also yields a tail estimate for the time-uniform pathwise distance. The proof uses stopped quasi-martingales and removes the stopping only at the final step.

For an integrable adapted c\`adl\`ag process $Q=(Q_t)_{0\le t\le T}$, define its mean variation by
\[
  V_T(Q):=\sup_\pi\sum_{i=0}^{m-1}
  \E\left|\E\left[Q_{t_{i+1}}-Q_{t_i}\mid\F_{t_i}\right]\right|,
\]
where the supremum is taken over all partitions $\pi:0=t_0<t_1<\cdots<t_m=T$. If $V_T(Q)<\infty$, then $Q$ is called a quasi-martingale.

The following classical maximal inequality for nonnegative quasi-martingales is used; see Kurtz~\cite[Lemma~5.3]{Kurtz1991}, Protter~\cite[Chapter~III]{Protter2004}, and Meyer--Zheng~\cite{MeyerZheng1984}. A short proof is included for completeness.

\begin{lemma}[Maximal inequality for nonnegative quasi-martingales]\label{lem:qm-max}
Let $R=(R_t)_{0\le t\le T}$ be a nonnegative c\`adl\`ag adapted process such that $R_t$ is integrable for each $t$ and $V_T(R)<\infty$. Then, for every $h>0$,
\[
  \Pbb\left(\sup_{0\le t\le T}R_t>h\right)
  \le \frac{V_T(R)+\E R_T}{h}.
\]
\end{lemma}

\begin{proof}
First consider the discrete-time version. Fix a partition $0=t_0<\cdots<t_m=T$ and set $R_i=R_{t_i}$. Put
\[
  h_i:=\E[R_i-R_{i+1}\mid\F_{t_i}],
  \qquad i=0,\dots,m-1,
\]
and define
\[
  Y_i:=\E\left[\sum_{j=i}^{m-1}h_j^+ + R_m\,\middle|\,\F_{t_i}\right].
\]
Then $Y$ is a nonnegative supermartingale, because
\[
  Y_i-\E[Y_{i+1}\mid\F_{t_i}]=h_i^+\ge0.
\]
Moreover $Y_i\ge R_i$, since $h_j^+\ge h_j$ and hence
\[
Y_i\ge \E\left[\sum_{j=i}^{m-1}h_j+R_m\,\middle|\,\F_{t_i}\right]=R_i.
\]
By Doob's maximal inequality for nonnegative supermartingales,
\[
  \Pbb\left(\max_i R_i>h\right)
  \le \Pbb\left(\max_i Y_i>h\right)
  \le \frac{\E Y_0}{h}.
\]
Finally,
\[
  \E Y_0=\sum_{j=0}^{m-1}\E h_j^+ + \E R_m
  \le \sum_{j=0}^{m-1}\E |h_j|+\E R_T
  \le V_T(R)+\E R_T.
\]
To pass to continuous time, choose a countable dense set $D\subset[0,T]$ containing $T$ and exhaust it by finite partitions. The c\`adl\`ag property implies $\sup_{t\in D}R_t=\sup_{0\le t\le T}R_t$, and the discrete estimate passes to the limit by monotone convergence.
\end{proof}

\begin{lemma}[Stopped mean variation bound]\label{lem:stopped-mean-variation}
For fixed $\delta\ge2$, $\eps\in(0,1)$ and $n\ge1$, define
\[
  Q_t^{(n)}:=u_{\delta,\eps}(Y_{t\wedge\tau_n}),
  \qquad 0\le t\le T.
\]
Then $Q^{(n)}$ is a nonnegative quasi-martingale, and
\begin{align}
V_T(Q^{(n)})+\E Q_T^{(n)}
\le C\Bigg[&|Y_0|^{\alpha-1}
+\eps^{\alpha-1}
+\frac{\delta}{\log\delta}\frac{B}{\eps^{2-\alpha}}
+\frac{\delta^2}{\log\delta}\frac{A^2}{\eps^{3-\alpha}} \notag\\
&+\frac{\delta}{\log\delta}\frac{S^\alpha}{\eps}
+\frac{\delta^2}{\log\delta}\eps^{\alpha-1}
+\frac{\eps^{\alpha\wt\eta-1}}{\log\delta}
\Bigg], \label{eq:stopped-mean-variation-bound}
\end{align}
where $C$ is independent of $n,\delta,\eps$.
\end{lemma}

\begin{proof}
From the stopped It\^o decomposition \eqref{eq:ito}, write
\[
  Q_t^{(n)}=Q_0+M_t^{(n)}+J_t^{(n)},
\]
where $M^{(n)}$ is a true martingale by Lemma~\ref{lem:localized-martingales}, and
\[
  J_t^{(n)}=\int_0^t H_s\1_{\{s\le\tau_n\}}\dd s
\]
contains the drift, Brownian correction and stable compensator densities. More explicitly,
\[
H_s=H_s^b+H_s^a+H_s^\sigma,
\]
where
\[
H_s^b=u'_{\delta,\eps}(Y_s)\bigl(b(s,X_s)-\wt b(s,\wt X_s)\bigr),
\]
\[
H_s^a=\frac12 u''_{\delta,\eps}(Y_s)\bigl(a(s,X_s)-\wt a(s,\wt X_s)\bigr)^2,
\]
and
\[
H_s^\sigma=C_\alpha|\sigma(s,X_s)-\wt\sigma(s,\wt X_s)|^\alpha
\psi_{\delta,\eps}(Y_s).
\]
For any partition,
\begin{align*}
\sum_i\E\left|\E\left[Q_{t_{i+1}}^{(n)}-Q_{t_i}^{(n)}\mid\F_{t_i}\right]\right|
&\le \sum_i\E\int_{t_i}^{t_{i+1}}|H_s|\1_{\{s\le\tau_n\}}\dd s\\
&=\E\int_0^{T\wedge\tau_n}|H_s|\dd s.
\end{align*}
The estimates of Section~\ref{sec:proof-fundamental} give separate bounds for the integrals of $|H_s^b|$, $|H_s^a|$, and $|H_s^\sigma|$. In particular,
\begin{align*}
V_T(Q^{(n)})
\le C\Bigg[&
\frac{\delta}{\log\delta}\frac{B}{\eps^{2-\alpha}}
+\frac{\delta^2}{\log\delta}\frac{A^2}{\eps^{3-\alpha}}
+\frac{\delta}{\log\delta}\frac{S^\alpha}{\eps}
+\frac{\delta^2}{\log\delta}\eps^{\alpha-1}
+\frac{\eps^{\alpha\wt\eta-1}}{\log\delta}
\Bigg]\\
&+C\int_0^T(\ell_b(s)+\ell_a(s)^2)\E|Y_{s\wedge\tau_n}|^{\alpha-1}\dd s.
\end{align*}
The stopped Gronwall estimate proved in Proposition~\ref{prop:fundamental-delta} controls the last integral uniformly in $n$. Moreover Lemma~\ref{lem:approx} and the same stopped estimate imply
\[
  \E Q_T^{(n)}\le C\left(\E|Y_{T\wedge\tau_n}|^{\alpha-1}+\eps^{\alpha-1}\right),
\]
which is bounded by the right-hand side of \eqref{eq:stopped-mean-variation-bound}. This proves the lemma.
\end{proof}

\begin{theorem}[Uniform convergence in probability]\label{thm:probability}
Assume Assumptions~\ref{ass:baseline}--\ref{ass:solutions} and suppose $B,A,S<1$. Then there exists a constant $C>0$, independent of $h$, such that, for every $h>0$, if $\wt\eta\in(1/\alpha,1]$, then
\[
\Pbb\left(\sup_{0\le t\le T}|X_t-\wt X_t|^{\alpha-1}>h\right)
\le \frac{C}{h}\left[|x_0-\wt x_0|^{\alpha-1}
+\max\left\{
B^{\frac{\alpha\wt\eta-1}{\alpha\wt\eta-\alpha+1}},
A^{\frac{2(\alpha\wt\eta-1)}{\alpha\wt\eta-\alpha+2}},
S^{\alpha-1/\wt\eta}
\right\}\right].
\]
If $\wt\eta=1/\alpha$ and $\Theta=\max\{B,A,S\}\in(0,1)$, then
\[
\Pbb\left(\sup_{0\le t\le T}|X_t-\wt X_t|^{\alpha-1}>h\right)
\le \frac{C}{h}\left[|x_0-\wt x_0|^{\alpha-1}
+\left(\log\frac1\Theta\right)^{-1}\right].
\]
When $\Theta=0$, the right-hand side is interpreted as $C|x_0-\wt x_0|^{\alpha-1}/h$.
\end{theorem}

\begin{proof}
In the non-critical case, if $\max\{B,A,S\}=0$, the estimate follows by taking $\delta=2$ and then $\eps\downarrow0$ in the stopped bound below. In the critical case with $\Theta=0$, the same stopped bound is used with $\delta=\max\{2,\eps^{-(\alpha-1)/4}\}$ and then $\eps\downarrow0$, as in the proof of Theorem~\ref{thm:critical}. Thus we may assume the optimizing parameter is positive. Fix $\delta\ge2$, $\eps\in(0,1)$ and $n\ge1$. By Lemma~\ref{lem:approx},
\[
  |Y_{t\wedge\tau_n}|^{\alpha-1}\le u_{\delta,\eps}(Y_{t\wedge\tau_n})+\eps^{\alpha-1}=Q_t^{(n)}+\eps^{\alpha-1}.
\]
Set $R_t^{(n)}:=Q_t^{(n)}+\eps^{\alpha-1}$. Then $R^{(n)}$ is nonnegative and
\[
  V_T(R^{(n)})=V_T(Q^{(n)}),
  \qquad
  \E R_T^{(n)}\le \E Q_T^{(n)}+\eps^{\alpha-1}.
\]
Lemma~\ref{lem:qm-max} and Lemma~\ref{lem:stopped-mean-variation} give
\begin{align*}
\Pbb\left(\sup_{0\le t\le T}|Y_{t\wedge\tau_n}|^{\alpha-1}>h\right)
&\le \Pbb\left(\sup_{0\le t\le T}R_t^{(n)}>h\right)\\
&\le \frac{C}{h}\Bigg[|Y_0|^{\alpha-1}
+\eps^{\alpha-1}
+\frac{\delta}{\log\delta}\frac{B}{\eps^{2-\alpha}}
+\frac{\delta^2}{\log\delta}\frac{A^2}{\eps^{3-\alpha}}\\
&\hspace{3.4cm}
+\frac{\delta}{\log\delta}\frac{S^\alpha}{\eps}
+\frac{\delta^2}{\log\delta}\eps^{\alpha-1}
+\frac{\eps^{\alpha\wt\eta-1}}{\log\delta}
\Bigg].
\end{align*}
Since $\tau_n=T$ eventually almost surely,
\[
  \sup_{0\le t\le T}|Y_{t\wedge\tau_n}|
  =\sup_{0\le t\le \tau_n}|Y_t|
  \uparrow \sup_{0\le t\le T}|Y_t|
  \qquad\text{a.s.}
\]
Thus the events on the left increase to the event with the unstopped supremum, and monotone convergence of probabilities removes the stopping. The choices of $(\delta,\eps)$ used in the proofs of Theorems~\ref{thm:noncritical} and~\ref{thm:critical} give the asserted rates.
\end{proof}

\section{Stability with predictable forcing errors}
\label{sec:forcing}

The proof of Proposition~\ref{prop:fundamental-delta} separates direct coefficient errors from path-difference errors. Therefore the same argument also treats predictable perturbations that are not functions of the current state. This formulation applies to error terms produced, for example, by a numerical scheme or by an exogenous perturbation.

Let $\beta$, $\varrho$ and $\zeta$ be predictable real-valued processes and suppose that the perturbed equation is replaced by
\begin{align}
\wt X_t
&=\wt x_0+\int_0^t\{\wt b(s,\wt X_s)+\beta_s\}\dd s
+\int_0^t\{\wt a(s,\wt X_s)+\varrho_s\}\dd W_s \notag\\
&\qquad +\int_0^t\{\wt\sigma(s,\wt X_{s-})+\zeta_s\}\dd Z_s.
\label{eq:forced-perturbed}
\end{align}
Define
\[
  \mathcal B:=\int_0^T\E|\beta_s|\dd s,
  \qquad
  \mathcal A:=\left(\int_0^T\E|\varrho_s|^2\dd s\right)^{1/2},
  \qquad
  \mathcal S:=\left(\int_0^T\E|\zeta_s|^\alpha\dd s\right)^{1/\alpha}.
\]
Set
\[
  \overline B:=B+\mathcal B,
  \qquad
  \overline A:=\bigl(A^2+\mathcal A^2\bigr)^{1/2},
  \qquad
  \overline S:=\bigl(S^\alpha+\mathcal S^\alpha\bigr)^{1/\alpha}.
\]

\begin{proposition}[Predictable forcing errors]\label{prop:forcing}
Assume Assumptions~\ref{ass:baseline}--\ref{ass:solutions}, with \eqref{eq:Xtilde} replaced by \eqref{eq:forced-perturbed}. Assume that $\mathcal B,\mathcal A,\mathcal S<\infty$ and that the stochastic integrals in \eqref{eq:forced-perturbed} are well defined. Then Proposition~\ref{prop:fundamental-delta} remains valid with $B,A,S$ replaced by $\overline B,\overline A,\overline S$. Consequently, the optimized estimates of Theorems~\ref{thm:noncritical},~\ref{thm:critical}, and~\ref{thm:probability} remain valid with the same replacement whenever their smallness assumptions are satisfied; that is, $\overline B,\overline A,\overline S<1$ in the non-critical case, and $\overline\Theta:=\max\{\overline B,\overline A,\overline S\}\in(0,1)$ in the critical case, with the same limiting convention when $\overline\Theta=0$.
\end{proposition}

\begin{proof}
Only the three direct-difference estimates in the proof of Proposition~\ref{prop:fundamental-delta} change. The drift difference is
\[
  b(s,X_s)-\wt b(s,\wt X_s)-\beta_s
  =\{b(s,X_s)-\wt b(s,X_s)\}
   +\{\wt b(s,X_s)-\wt b(s,\wt X_s)\}-\beta_s.
\]
The first and third terms are controlled by $B+\mathcal B$ using the bound for $\|u'_{\delta,\eps}\|_\infty$, while the middle term is unchanged and gives the recursive term involving $\ell_b(s)\E|Y_{s\wedge\tau_n}|^{\alpha-1}$.

For the Brownian correction,
\[
  a(s,X_s)-\wt a(s,\wt X_s)-\varrho_s
  =\Delta_s^a+\Gamma_s^a-\varrho_s,
\]
where $\Delta_s^a=a(s,X_s)-\wt a(s,X_s)$ and $\Gamma_s^a=\wt a(s,X_s)-\wt a(s,\wt X_s)$. The inequality $|x+y+z|^2\le3(|x|^2+|y|^2+|z|^2)$ shows that the direct part is bounded by $A^2+\mathcal A^2$, up to a numerical constant, while the path-difference part is unchanged. Thus the Brownian term has the same form with $A$ replaced by $\overline A$.

Finally, for the stable compensator,
\[
  \sigma(s,X_s)-\wt\sigma(s,\wt X_s)-\zeta_s
  =\Delta_s^\sigma+\Gamma_s^\sigma-\zeta_s.
\]
Using $|x+y+z|^\alpha\le 3^{\alpha-1}(|x|^\alpha+|y|^\alpha+|z|^\alpha)$, the direct part is bounded by $S^\alpha+\mathcal S^\alpha$ and the H\"older path-difference part is unchanged. All remaining localization, martingale, and Gronwall steps are identical to those in Sections~\ref{sec:localization} and~\ref{sec:proof-fundamental}. The optimized expectation and probability estimates then follow by applying Theorems~\ref{thm:noncritical},~\ref{thm:critical}, and~\ref{thm:probability} to the barred quantities under the smallness hypotheses stated in those theorems.
\end{proof}

\section{Further consequences}
\label{sec:consequences}

\paragraph{Uniform coefficient errors.}
The law-weighted distances can be controlled by spatial supremum norms on two useful time scales.
Define the time-integrated quantities
\[
B_{1,\infty}
:=\|b-\wt b\|_{L^1(0,T;L^\infty(\R))},
\qquad
A_{2,\infty}
:=\|a-\wt a\|_{L^2(0,T;L^\infty(\R))},
\]
and
\[
S_{\alpha,\infty}
:=\|\sigma-\wt\sigma\|_{L^\alpha(0,T;L^\infty(\R))}.
\]
Also define the time-uniform quantities
\[
\hat B_\infty
:=\|b-\wt b\|_{L^\infty(0,T;L^\infty(\R))},
\qquad
\hat A_\infty
:=\|a-\wt a\|_{L^\infty(0,T;L^\infty(\R))},
\]
and
\[
\hat S_\infty
:=\|\sigma-\wt\sigma\|_{L^\infty(0,T;L^\infty(\R))},
\]
where the $L^\infty(0,T)$ norm is the essential supremum in time.
Then
\[
B\le B_{1,\infty},\qquad
A\le A_{2,\infty},\qquad
S\le S_{\alpha,\infty},
\]
while
\[
B\le T\hat B_\infty,\qquad
A\le T^{1/2}\hat A_\infty,\qquad
S\le T^{1/\alpha}\hat S_\infty.
\]
For each coefficient class, the time scale may be selected independently.
Accordingly, the selected triple of error norms may be any element of the Cartesian product
\[
\{B_{1,\infty},\hat B_\infty\}
\times
\{A_{2,\infty},\hat A_\infty\}
\times
\{S_{\alpha,\infty},\hat S_\infty\}.
\]
This product consists of exactly $2^3=8$ admissible triples; in particular, mixed choices such as $(\hat B_\infty,A_{2,\infty},\hat S_\infty)$ are allowed, and the same time norm need not be used for all three coefficient errors.
By repeating the parameter optimizations in Theorems~\ref{thm:noncritical} and~\ref{thm:critical} and the stopped quasi-martingale argument in Theorem~\ref{thm:probability}, we obtain the same rate exponents with $B,A,S$ replaced by the corresponding components of any selected triple of error norms, whenever those components are less than one.
Whenever a time-uniform quantity is selected, the corresponding factor $T$, $T^{1/2}$, or $T^{1/\alpha}$ is absorbed into the constant.

Neither choice dominates the other in any coefficient class.
More precisely, when $T>1$, for $p\in\{1,2,\alpha\}$ a persistent discrepancy $d(t)\equiv c$ with $T^{-1/p}<c<1$ satisfies $\|d\|_{L^\infty(0,T)}<1<\|d\|_{L^p(0,T)}$, whereas a short spike $d(t)=M\1_{[0,\varepsilon]}(t)$ with $M>1$ and $0<\varepsilon<\min\{T,M^{-p}\}$ satisfies $\|d\|_{L^p(0,T)}<1<\|d\|_{L^\infty(0,T)}$.
This comparison applies independently with $p=1$ for the drift, $p=2$ for the Brownian diffusion, and $p=\alpha$ for the stable jump coefficient.
Thus integrated and time-uniform norms may be combined according to the temporal profile of each individual coefficient perturbation.

\subsection{Time-homogeneous approximation}
\label{sec:approximation}

The next corollary is a direct specialization of Theorems~\ref{thm:noncritical} and~\ref{thm:critical} to time-homogeneous coefficient sequences driven by the same noises. It assumes the existence of the coupled solutions and convergence of the corresponding law-weighted coefficient distances.

\begin{corollary}[Coefficient approximation]
\label{cor:approx}
Let $X$ solve the baseline equation \eqref{eq:X}. For each $n\ge1$, let $X^{(n)}$ be a c\`adl\`ag adapted solution, driven by the same pair $(W,Z)$, such that the pair $(X,X^{(n)})$ satisfies Assumption~\ref{ass:solutions}. Suppose that $X^{(n)}$ solves
\[
X_t^{(n)}=x_0^{(n)}+\int_0^t b_n(X_s^{(n)})\dd s
+\int_0^t a_n(X_s^{(n)})\dd W_s
+\int_0^t \sigma_n(X_{s-}^{(n)})\dd Z_s.
\]
Assume that there exist common constants $L_b,L_a,L_\sigma$ and exponents $\eta_n\in[\wt\eta,1]$, with $\wt\eta\in[1/\alpha,1]$, such that for all $n$ and all $x,y\in\R$,
\[
|b_n(x)-b_n(y)|\le L_b|x-y|,
\qquad
|a_n(x)-a_n(y)|\le L_a|x-y|,
\]
and
\[
|\sigma_n(x)-\sigma_n(y)|\le L_\sigma\bigl(|x-y|^{\eta_n}\vee |x-y|\bigr).
\]
Then the family satisfies Assumption~\ref{ass:perturbed} uniformly with exponent $\wt\eta$. Define
\begin{align*}
B_n&:=\int_0^T\E |b(s,X_s)-b_n(X_s)|\dd s,\\
A_n&:=\left(\int_0^T\E |a(s,X_s)-a_n(X_s)|^2\dd s\right)^{1/2},\\
S_n&:=\left(\int_0^T\E |\sigma(s,X_s)-\sigma_n(X_s)|^\alpha\dd s\right)^{1/\alpha}.
\end{align*}
Suppose that $B_n,A_n,S_n\to0$ and $x_0^{(n)}\to x_0$. Then
\[
\sup_{0\le t\le T}\E|X_t-X_t^{(n)}|^{\alpha-1}\to0.
\]
Moreover, once $B_n,A_n,S_n<1$, the rate is obtained from Theorem~\ref{thm:noncritical} if $\wt\eta>1/\alpha$ and from Theorem~\ref{thm:critical} if $\wt\eta=1/\alpha$.
\end{corollary}

\begin{proof}
Apply Theorem~\ref{thm:noncritical} or Theorem~\ref{thm:critical} to the pair $(X,X^{(n)})$, with perturbed coefficients $(b_n,a_n,\sigma_n)$. The uniform modulus assumptions ensure that Assumption~\ref{ass:perturbed} holds with exponent $\wt\eta$, and the quantities $B_n,A_n,S_n$ are exactly the corresponding law-weighted coefficient distances. The convergence of the initial values and of $B_n,A_n,S_n$ gives the expectation convergence; the stated rates are the same optimized rates with $B,A,S$ replaced by $B_n,A_n,S_n$.
\end{proof}

\begin{remark}[Discretization errors as predictable forcing]
For Euler--Maruyama schemes, let $X^h$ denote the approximation with time step $h$, and set $\pi_h(s):=h\lfloor s/h\rfloor$. The frozen arguments $X_{\pi_h(s)}^{h}$ produce predictable forcing errors. Proposition~\ref{prop:forcing} gives the corresponding stability estimate once the discretization errors
\[
  b(s,X_s^h)-b(s,X_{\pi_h(s)}^h),\quad
  a(s,X_s^h)-a(s,X_{\pi_h(s)}^h),\quad
  \sigma(s,X_s^h)-\sigma(s,X_{\pi_h(s)}^h)
\]
are estimated in the corresponding $L^1$, $L^2$, and $L^\alpha$ scales.
\end{remark}

\section{Concluding remarks}
\label{sec:concluding}

\begin{remark}[Brownian $1/2$-H\"older coefficient]
For Brownian SDEs alone, the Yamada-Watanabe threshold for the diffusion coefficient is $1/2$-H\"older continuity. Under a condition of the form
\[
  |\wt a(t,x)-\wt a(t,y)|\le C|x-y|^{1/2},
\]
the Brownian path-difference term contains a factor of the order
\[
  |u_{\delta,\eps}''(Y_s)|\,|Y_s|,
\]
which behaves like $|Y_s|^{\alpha-2}$ near the origin.
This factor has a stronger singularity at the origin than the present $|Y_s|^{\alpha-1}$ comparison scale. A comparison estimate at the simultaneous Brownian $1/2$-H\"older and stable $1/\alpha$-H\"older thresholds would require a localization scheme adapted to this two-threshold structure.
\end{remark}

\begin{remark}[Density representation]
The estimates are stated in the law-weighted form \eqref{eq:B-exp}-\eqref{eq:S-exp}. The proof uses these expectations along the baseline process. When $X$ admits a transition density, the identities \eqref{eq:B-density}-\eqref{eq:S-density} express the same quantities as spatially weighted norms. In the stable-only case, such density representations follow under boundedness, positivity, and H\"older-type assumptions on the stable coefficient from the parametrix framework of Knopova and Kulik \cite{KnopovaKulik2018}. Thus density estimates give spatial-norm formulations of the same law-weighted quantities.
\end{remark}

\begin{remark}[Role of localization]
The estimates are proved by localization and do not use deterministic envelopes of the perturbed coefficients. On stopped intervals, the decompositions through $B,A,S$ and the moduli $\ell_b,\ell_a,\ell_\sigma$ give the required integrability. Consequently, the constants in the final estimates depend on $\|\ell_b\|_{L^1(0,T)}$, $\|\ell_a\|_{L^2(0,T)}$ and $\|\ell_\sigma\|_{L^\alpha(0,T)}$, with no additional global-envelope parameters.
\end{remark}

\begin{remark}[Non-Lipschitz exponents]
The consistency examples in \Cref{sec:consistency} identify the sharp scale in the Lipschitz constant-coefficient subclass. For $\wt\eta<1$, the rate exponents arise from the Komatsu localization and the mixed balancing of $B,A,S$. Propositions~\ref{prop:nonlipschitz-example},~\ref{prop:mixed-brownian-error},~\ref{prop:all-errors-example}, and~\ref{prop:mixed-tail-brownian-error} give explicit H\"older examples, including a case in which all three coefficient distances occur, cases where the estimate reduces to the Brownian diffusion error, and a case where the finiteness of that distance is verified through a stable moment condition. Lower-bound estimates in such H\"older examples require coefficient-specific local constructions.
\end{remark}

\section*{Statements and Declarations}

\paragraph{Funding.}
This work was supported by JSPS KAKENHI Grant Number 23K12507.

\paragraph{Competing interests.}
The authors declare no competing interests.

\paragraph{Data availability.}
No datasets were generated or analyzed during the current study.

\paragraph{Author contributions.}
T. Nakagawa and R. Suzuki conceived the study, developed the mathematical analysis, and wrote the manuscript. Both authors read and approved the final manuscript.

\bibliographystyle{abbrvurl}
\bibliography{references}

@book{Applebaum2009,
  author    = {Applebaum, David},
  title     = {L{\'e}vy Processes and Stochastic Calculus},
  edition   = {Second},
  publisher = {Cambridge University Press},
  address   = {Cambridge},
  year      = {2009},
  doi       = {10.1017/CBO9780511809781},
  url       = {https://doi.org/10.1017/CBO9780511809781}
}

@article{Komatsu1982,
  author  = {Komatsu, Takashi},
  title   = {On the pathwise uniqueness of solutions of one-dimensional stochastic differential equations of jump type},
  journal = {Proceedings of the Japan Academy, Series A, Mathematical Sciences},
  volume  = {58},
  number  = {8},
  pages   = {353--356},
  year    = {1982},
  doi     = {10.3792/pjaa.58.353},
  url     = {https://doi.org/10.3792/pjaa.58.353}
}

@article{KnopovaKulik2018,
  author  = {Knopova, Victoria and Kulik, Alexei},
  title   = {Parametrix construction of the transition probability density of the solution to an {SDE} driven by {$\alpha$}-stable noise},
  journal = {Annales de l'Institut Henri Poincar{\'e}, Probabilit{\'e}s et Statistiques},
  volume  = {54},
  number  = {1},
  pages   = {100--140},
  year    = {2018},
  doi     = {10.1214/16-AIHP796},
  url     = {https://doi.org/10.1214/16-AIHP796}
}

@article{Nakagawa2020,
  author  = {Nakagawa, Takuya},
  title   = {{$L^{\alpha-1}$} distance between two one-dimensional stochastic differential equations driven by a symmetric {$\alpha$}-stable process},
  journal = {Japan Journal of Industrial and Applied Mathematics},
  volume  = {37},
  number  = {3},
  pages   = {929--956},
  year    = {2020},
  doi     = {10.1007/s13160-020-00429-9},
  url     = {https://doi.org/10.1007/s13160-020-00429-9}
}

@article{Nakagawa2026,
  author  = {Nakagawa, Takuya},
  title   = {{$L^{\alpha-1}$} distance between two one-dimensional stochastic differential equations with drift terms driven by a symmetric {$\alpha$}-stable process},
  journal = {Japan Journal of Industrial and Applied Mathematics},
  volume  = {43},
  pages   = {45},
  year    = {2026},
  doi     = {10.1007/s13160-026-00800-2},
  url     = {https://doi.org/10.1007/s13160-026-00800-2}
}

@article{BassBurdzyChen2004,
  author  = {Bass, Richard F. and Burdzy, Krzysztof and Chen, Zhen-Qing},
  title   = {Stochastic differential equations driven by stable processes for which pathwise uniqueness fails},
  journal = {Stochastic Processes and their Applications},
  volume  = {111},
  number  = {1},
  pages   = {1--15},
  year    = {2004},
  doi     = {10.1016/j.spa.2004.01.010},
  url     = {https://doi.org/10.1016/j.spa.2004.01.010}
}

@article{ChenKimSong2010,
  author  = {Chen, Zhen-Qing and Kim, Panki and Song, Renming},
  title   = {Heat kernel estimates for {$\Delta+\Delta^{\alpha/2}$} in {$C^{1,1}$} open sets},
  journal = {Journal of the London Mathematical Society},
  volume  = {84},
  number  = {1},
  pages   = {58--80},
  year    = {2011},
  doi     = {10.1112/jlms/jdq102},
  url     = {https://doi.org/10.1112/jlms/jdq102}
}

@article{ChenHu2015,
  author  = {Chen, Zhen-Qing and Hu, Eryan},
  title   = {Heat kernel estimates for {$\Delta+\Delta^{\alpha/2}$} under gradient perturbation},
  journal = {Stochastic Processes and their Applications},
  volume  = {125},
  number  = {7},
  pages   = {2603--2642},
  year    = {2015},
  doi     = {10.1016/j.spa.2015.02.016},
  url     = {https://doi.org/10.1016/j.spa.2015.02.016}
}

@article{YamadaWatanabe1971,
  author  = {Yamada, Toshio and Watanabe, Shinzo},
  title   = {On the uniqueness of solutions of stochastic differential equations},
  journal = {Journal of Mathematics of Kyoto University},
  volume  = {11},
  number  = {1},
  pages   = {155--167},
  year    = {1971},
  doi     = {10.1215/kjm/1250523691},
  url     = {https://doi.org/10.1215/kjm/1250523691}
}

@article{Kurtz1991,
  author  = {Kurtz, Thomas G.},
  title   = {Random time changes and convergence in distribution under the Meyer-Zheng conditions},
  journal = {Annals of Probability},
  volume  = {19},
  number  = {3},
  pages   = {1010--1034},
  year    = {1991},
  doi     = {10.1214/aop/1176990422},
  url     = {https://doi.org/10.1214/aop/1176990422}
}

@article{Fournier2013,
  author  = {Fournier, Nicolas},
  title   = {On pathwise uniqueness for stochastic differential equations driven by stable {L{\'e}vy} processes},
  journal = {Annales de l'IHP Probabilit{\'e}s et Statistiques},
  volume  = {49},
  number  = {1},
  pages   = {138--159},
  year    = {2013},
  doi     = {10.1214/11-AIHP420},
  url     = {https://doi.org/10.1214/11-AIHP420}
}

@book{IkedaWatanabe1989,
  author    = {Ikeda, Nobuyuki and Watanabe, Shinzo},
  title     = {Stochastic Differential Equations and Diffusion Processes},
  edition   = {Second},
  publisher = {North-Holland},
  address   = {Amsterdam},
  year      = {1989}
}

@book{GelfandShilov1964,
  author    = {Gel'fand, Izrail M. and Shilov, Georgii E.},
  title     = {Generalized Functions, Vol. 1: Properties and Operations},
  publisher = {Academic Press},
  address   = {New York},
  year      = {1964}
}

@book{Landkof1972,
  author    = {Landkof, N. S.},
  title     = {Foundations of Modern Potential Theory},
  publisher = {Springer},
  address   = {Berlin},
  year      = {1972},
  doi       = {10.1007/978-3-642-65183-0},
  url       = {https://doi.org/10.1007/978-3-642-65183-0}
}

@article{Kwasnicki2017,
  author  = {Kwa\'snicki, Mateusz},
  title   = {Ten equivalent definitions of the fractional {L}aplace operator},
  journal = {Fractional Calculus and Applied Analysis},
  volume  = {20},
  number  = {1},
  pages   = {7--51},
  year    = {2017},
  doi     = {10.1515/fca-2017-0002},
  url     = {https://doi.org/10.1515/fca-2017-0002}
}

@incollection{Hashimoto2013,
  author    = {Hashimoto, Hiroya},
  title     = {Approximation and stability of solutions of {SDEs} driven by a symmetric {$\alpha$}-stable process with non-{L}ipschitz coefficients},
  booktitle = {S\'eminaire de Probabilit\'es XLV},
  series    = {Lecture Notes in Mathematics},
  volume    = {2078},
  pages     = {181--199},
  publisher = {Springer},
  address   = {Cham},
  year      = {2013},
  doi       = {10.1007/978-3-319-00321-4_7},
  url       = {https://doi.org/10.1007/978-3-319-00321-4_7}
}

@misc{HashimotoTsuchiya2014,
  author       = {Hashimoto, Hiroya and Tsuchiya, Takahiro},
  title        = {Convergence rate of stability problems of {SDEs} with (dis-)continuous coefficients},
  year         = {2014},
  eprint       = {1401.4542},
  archivePrefix= {arXiv},
  primaryClass = {math.PR},
  url          = {https://arxiv.org/abs/1401.4542}
}

@article{Priola2012,
  author  = {Priola, Enrico},
  title   = {Pathwise uniqueness for singular {SDEs} driven by stable processes},
  journal = {Osaka Journal of Mathematics},
  volume  = {49},
  number  = {2},
  pages   = {421--447},
  year    = {2012},
  url     = {https://projecteuclid.org/journals/osaka-journal-of-mathematics/volume-49/issue-2/Pathwise-uniqueness-for-singular-SDEs-driven-by-stable-processes/ojm/1340197933.full}
}

@article{Qiao2014,
  author  = {Qiao, Huijie},
  title   = {{Euler--Maruyama} approximation for {SDEs} with jumps and non-{L}ipschitz coefficients},
  journal = {Osaka Journal of Mathematics},
  volume  = {51},
  number  = {1},
  pages   = {47--66},
  year    = {2014},
  url     = {https://projecteuclid.org/euclid.ojm/1396966224}
}

@incollection{Bass2002,
  author    = {Bass, Richard F.},
  title     = {Stochastic differential equations driven by symmetric stable processes},
  booktitle = {S\'eminaire de Probabilit\'es XXXVI},
  series    = {Lecture Notes in Mathematics},
  volume    = {1801},
  pages     = {302--313},
  publisher = {Springer},
  address   = {Berlin},
  year      = {2003},
  url       = {https://www.numdam.org/item/SPS_2002__36__302_0/}
}

@book{Protter2004,
  author    = {Protter, Philip E.},
  title     = {Stochastic Integration and Differential Equations},
  edition   = {Second},
  publisher = {Springer},
  address   = {Berlin},
  year      = {2004}
}

@article{MeyerZheng1984,
  author  = {Meyer, Paul-Andr\'e and Zheng, Wei-An},
  title   = {Tightness criteria for laws of semimartingales},
  journal = {Annales de l'I.H.P. Probabilit{\'e}s et Statistiques},
  volume  = {20},
  number  = {4},
  pages   = {353--372},
  year    = {1984},
  url     = {https://www.numdam.org/item/AIHPB_1984__20_4_353_0/}
}

\end{document}